\documentclass[a4paper,11pt]{article}

\usepackage[utf8]{inputenc}
\usepackage[T1]{fontenc}
\usepackage{graphicx}
\usepackage{cancel}
\usepackage{amsmath,amsfonts,amssymb,amsthm}
\usepackage{mathtools}
\usepackage{mathrsfs}
\usepackage{xcolor}
\usepackage{hyperref}
\usepackage{booktabs}
\usepackage{bbold}
\usepackage{bm}
\usepackage{color}
\usepackage[font=footnotesize]{caption}
\usepackage{enumitem}
\usepackage{empheq}
\usepackage{framed}
\usepackage{fullpage}
\usepackage{hyperref}
\usepackage{soul}
\usepackage{dirtytalk}
\usepackage{pstricks}
\usepackage{appendix}
\usepackage{dsfont}
\usepackage{tcolorbox}
\usepackage{bbm}

\mathtoolsset{showonlyrefs}

\definecolor{myred}{HTML}{880000}
\definecolor{mygreen}{HTML}{008800}
\definecolor{myblue}{HTML}{000088}
\definecolor{linkblue}{HTML}{0000BB}

\hypersetup{
  colorlinks,
  linkcolor={myred},
  citecolor={myblue},
  urlcolor={linkblue}
}

\newcommand{\E}{{\mathbb E}}

\newcommand{\Var}{\operatorname{Var}}

\renewcommand{\leq}{\leqslant}
\renewcommand{\epsilon}{\varepsilon}

\renewcommand{\le}{\leqslant}
\renewcommand{\ge}{\geqslant}

\newcommand{\argmin}{\mathop{\mathrm{argmin}}}

\DeclareMathOperator{\tr}{Tr}

\newcommand{\var}{\Var}

\newcommand{\ind}[1]{\mathbbm 1 \{ #1 \}}

\renewcommand{\Pr}{\mathbb{P}}

\renewcommand{\top}{\mathsf{T}}
\DeclareMathOperator\supp{supp}

\usepackage{xspace}

\newtheorem{proposition}{Proposition}
\newtheorem{theorem}{Theorem}
\newtheorem*{theorem*}{Theorem}
\newtheorem{lemma}{Lemma}
\newtheorem{corollary}{Corollary}

\theoremstyle{definition}
\newtheorem{definition}{Definition}

\newtheorem{example}{Example}

\theoremstyle{remark}
\newtheorem{remark}{Remark}

\title{Covariance Estimation: Optimal Dimension-free Guarantees for Adversarial Corruption and Heavy Tails}
\author{
  Pedro Abdalla
  \qquad
  Nikita Zhivotovskiy\thanks{Department of Mathematics, ETH Z\"{u}rich, Switzerland, \href{mailto:pedro.abdallateixeira@math.ethz.ch}{pedro.abdallateixeira@math.ethz.ch},  \href{mailto:nikita.zhivotovskii@math.ethz.ch}{nikita.zhivotovskii@math.ethz.ch}}
}

\begin{document}

\maketitle

\begin{abstract}
We provide an estimator of the covariance matrix that achieves the optimal rate of convergence (up to constant factors) in the operator norm under two standard notions of data contamination: We allow the adversary to corrupt an $\eta$-fraction of the sample arbitrarily, while the distribution of the remaining data points only satisfies that the $L_{p}$-marginal moment with some $p \ge 4$ is equivalent to the corresponding $L_2$-marginal moment. Despite requiring the existence of only a few moments of the distribution, our estimator achieves the same tail estimates as if the underlying distribution were Gaussian. As a part of our analysis, we prove a non-asymptotic, dimension-free Bai-Yin type theorem in the regime $p > 4$.
\end{abstract}

\section{Introduction}

Estimation of the covariance matrix is a classic topic. In high-dimensional statistics the role of sample covariance matrices is central to Principal Component Analysis (PCA) and to linear least squares. Most of the existing work  focuses on estimation of covariance matrices under different structural assumptions allowing minimax estimation in the high-dimensional setup. We refer to the line of work \cite{bickel2008covariance, bickel2008regularized, lam2009sparsistency, el2008operator, cai2010optimal, cai2013optimal} and the recent surveys \cite{cai2016estimating, el2018random}.  

A line of research on non-asymptotic guarantees for sample covariance matrices was initiated by a question of Kannan, Lov{\'a}sz and Simonovits \cite{kannan1997random} on the computation of the volume of a convex body. For the class of log-concave measures the optimal $\sqrt{\frac{d}{N}}$ ($d$ is the dimension and $N$ is the sample size) rate of convergence in the operator norm was first obtained in the renowned work of Adamczak, Litvak, Pajor and Tomczak-Jaegermann \cite{adamczak2010quantitative}. Since then several authors focused on making less assumptions on the underlying distribution \cite{srivastava2013covariance, mendelson2012generic, mendelson2014singular, guedon2017interval}. The best known result in this direction is due to K. Tikhomirov \cite{tikhomirov2018sample} who proved the optimal rate of convergence $\sqrt{\frac{d}{N}}$ for the sample covariance matrix assuming only the existence of $p > 4$ moments. 

However, as discussed by Chen, Gao and Ren \cite{chen2018robust}, if 
there exists only one outlier in the whole sample, the statistical performance of the sample covariance matrix can be compromised. Thus, one is interested in estimators of the covariance matrix robust to adversarial contamination of the data \cite{cheng2019faster, minsker2022robust}. For a standard perspective on robustness,
building on the ideas of contaminated models, influence functions and breakdown points, we refer to the monographs \cite{hampel1971general, huber1981robust, rousseeuw1987robust}.

On the other hand, there is a growing interest in getting dimension-free guarantees for estimating the covariance matrix. Given a covariance matrix $\Sigma$, the \emph{effective rank} (see \cite{vershynin_2012}) of $\Sigma$ is defined as
\[
\mathbf{r}(\Sigma) = \frac{\tr(\Sigma)}{\|\Sigma\|},
\]
where, for the rest of the paper, $\|\cdot\|$ denotes the operator norm of the matrix and the Euclidean norm of the vector. Koltchinskii and Lounici \cite{koltchinskii2017operators} proved the optimal high probability bound for the sample covariance matrix in the Gaussian case that depends on the effective rank rather than the dimension. Since then, their result was recovered and extended multiple times and via different techniques \cite{van2017structured,Vershynin2016HDP, zhivotovskiy2021dimension, koltchinskii2020asymptotically}.

Finally, we mention the recent interest in getting the so-called \emph{sub-Gaussian estimators} when the data is heavy-tailed. This direction was initiated by  O. Catoni in \cite{catoni2012challenging}, where the sub-Gaussian estimation of the mean of a random variable is considered. To explain informally, one
aims to construct statistical estimators performing as good as the sample mean does for the Gaussian distribution, while making as weak assumptions on the distribution as possible. For a recent survey with focus on multivariate mean estimation, we refer to \cite{lugosi2019mean}. The central ideas behind the robust mean estimation found their applications in many related
problems such as regression \cite{hsu2016heavy, brownlees2015empirical, lugosi2019risk, chinot2019robust, mendelson2019unrestricted, mourtada2022distribution}, covariance estimation \cite{catoni2016pac, catoni2017dimension, mendelson2020robust, ostrovskii2019affine, depersin2020robust, minsker2022robust},
and clustering \cite{klochkov2021robust}. For related results in the context of covariance estimation for heavy-tailed distributions, we refer to the recent survey \cite{ke2019user}.

Our goal is to provide an estimator that simultaneously achieves all the properties described above:
\begin{itemize}
\item We allow the adversarial contamination and recover the optimal dependence on the contamination level based on the number of moments of the underlying distribution.
\item Our bounds do not contain unnecessary logarithmic factors and the convergence rates coincide with the classical asymptotic result of Bai and Yin \cite{baiyin1994} provided that there are at least four moments of the distribution.
\item The convergence rates scale with the effective rank $\mathbf{r}(\Sigma)$ rather than the dimension $d$.
\item We allow the distributions satisfying certain weak norm equivalence assumptions instead of more restrictive Gaussian/log-concave assumptions appearing in the literature. At the same time, we provide the same high probability bounds as if the data were Gaussian. 
\end{itemize}

We begin with the following definition. We say that the distribution of a zero mean random vector $X$ satisfies the $L_p-L_2$ \emph{norm equivalence} (\emph{hypercontractivity}), if for all $v \in \mathbb{R}^d$ and $2 \le q \le p$,
\begin{equation}
\label{eq:momeqv}
(\E|\langle X, v\rangle|^{q})^{1/q} \le \kappa(q)(\E|\langle X, v\rangle|^{2})^{1/2},
\end{equation}
where $\kappa(\cdot)$ is a function of $q$ and $\langle\cdot, \cdot\rangle$ denotes the inner product. Without loss of generality, we assume that $\kappa$ is a non-decreasing function. We say that the sample $\widetilde X_1, \ldots, \widetilde X_N$ is $\eta$-\emph{corrupted} if it is obtained from the sample $X_1, \ldots X_N$ of independent copies of $X$ by replacing at most $\eta N$ points by arbitrary vectors that might depend on $X_{1}, \ldots, X_N$. This corruption model is described in detail \cite{lugosi2021robust} and captures the standard setups in robust statistics such as the Huber contamination model \cite{huber1964}. This model is sometimes called the model of $\eta$-\emph{corruption} or the \emph{strong contamination model} \cite{diakonikolas2019recent}. We discuss this model in more detail in Section \ref{sec:estimtraceandopernorm}.
We present a simplified version of our main result.
\begin{theorem}[Simplified]
\label{thm:informalmain}
Assume that $X$ is a zero mean random vector in $\mathbb{R}^d$ with covariance $\Sigma$ satisfying $L_p-L_2$ \emph{norm equivalence} with $p \ge 4$. Fix the corruption level $\eta \in [0, 1]$ and the confidence level $\delta \in (0, 1)$. Assume that $\widetilde{X}_{1}, \ldots, \widetilde{X}_N$ is an $\eta$-corrupted sample. There is an estimator $\widehat{\Sigma}_{\eta, \delta} = \widehat{\Sigma}_{\eta, \delta}(\widetilde{X}_{1}, \ldots, \widetilde{X}_N)$ depending on $\eta, \delta$ such that if $N \ge c(p)(\mathbf{r}(\Sigma) + \log(1/\delta))$, then with probability at least $1 - \delta$, it holds that
\[
\left\|\widehat{\Sigma}_{\eta, \delta} - \Sigma\right\| \le C\|\Sigma\|\left(\sqrt{\frac{\mathbf{r}(\Sigma) + \log(1/\delta)}{N}} + \kappa(p)^2\eta^{1-2/p}\right),
\]
where $C > 0$ is an absolute constant that depends only on the value $\kappa(4)$ and $c(p)$ depends only on $p$ and $\kappa(p)$. Moreover, under these assumptions, no estimator can perform better (up to multiplicative constant factors).
\end{theorem}

We show that the term $\kappa(p)^2\eta^{1-2/p}$ scales as $\sqrt{\eta}$ when $p = 4$ and as $\eta\log(1/\eta)$ for sub-Gaussian distributions. We also show that both rates cannot be improved. In particular, in the special case of variance estimation, which corresponds to $d = 1$, our bound matches the recent rates of Comminges, Collier, Ndaoud, and Tsybakov \cite[Table 1]{comminges2021adaptive}, whose analysis covers a less general contamination model called the sparse vector model and only in-expectation bounds. A detailed version of the Theorem \ref{thm:informalmain} with the corresponding explicit estimators is stated in Theorem \ref{thm:thecasepfour} in the regime $p = 4$ and in Theorem \ref{thm:thecasep_greater_four} in the regime $p > 4$ respectively. For a detailed discussion on the optimality of our result, we refer to Section \ref{sec:optimality}. Since our estimator is robust to both heavy tails and adversarial corruption, it should necessarily differ from the sample covariance matrix. However, it still has a relatively simple form. The estimator depends on a specifically tuned scalar $\lambda > 0$, which in turn depends on some parameters including $\delta$ and $\eta$ (the details are postponed to Theorem \ref{thm:thecasepfour} and Theorem \ref{thm:thecasep_greater_four}). We define the truncation function
\begin{equation}
\label{eq:truncfunction}
\psi(x) = 
    \begin{cases}
      x,\quad &\textrm{for}\; x \in [-1, 1],
      \\
      \operatorname{sign}(x),\quad &\textrm{for}\; |x| > 1.
    \end{cases}
\end{equation}
Then, when $p = 4$, our estimator has the following form
\begin{equation}
\label{eq:minmaxestimator}
\widehat{\Sigma}_{\eta, \delta}= \argmin\limits_{\widehat{\Sigma} \in \mathbb{S}_{+}^d}\sup\limits_{v \in S^{d - 1}}\left|\frac{1}{\lambda N}\sum\limits_{i = 1}^N\psi(\lambda\langle \widetilde X_i, v\rangle^2) - v^{\top}\widehat{\Sigma} v \right|,
\end{equation}
where $\mathbb{S}_{+}^d$ is the set of $d$ by $d$ positive semi-definite matrices and $S^{d - 1}$ denotes the unit sphere in $\mathbb{R}^d$. A simple observation is that, if $\lambda \to 0$, our estimator coincides with the sample covariance matrix $\frac{1}{N}\sum\nolimits_{i = 1}^N\widetilde X_i\otimes \widetilde X_i$. For a fixed value of $\lambda$, our estimator can be seen as the second order extension of the classical trimmed mean estimator \cite{tukey1963less, bickel1965some, stigler1973asymptotic, lugosi2021robust}. We are aiming to remove a fraction of extreme observations and then average over the remaining sample. The only difference is that, in the matrix case, we need to take into account all possible directions in the unit sphere. The result of Theorem \ref{thm:informalmain} has multiple advantages over the best known results in the literature: 
\begin{itemize}
    \item  \textbf{Adversarial corruption:} One of the well known results in the adversarial corruption setup is due to Chen, Gao and Ren \cite{chen2018robust}. They analyze a version of Tukey's median under the restrictive assumption of elliptical distributions satisfying certain growth conditions. Moreover, their results are dimension dependent. See \cite{minsker2022robust} for some recent extensions.
    
    \item  \textbf{Heavy-tails:} The best-known result in the literature in the heavy-tailed setup (the only assumption is $L_4-L_2$ hypercontractivity) is due to S. Mendelson and the second-named author of this paper \cite{mendelson2020robust}. However, their results can be further improved. First, there is an additional logarithmic factor $\log \mathbf{r}(\Sigma)$ due to the application of the non-commutative Bernstein inequality in the analysis\footnote{We note that under the $L_4-L_2$ norm equivalence the logarithmic factor $\log \mathbf{r}(\Sigma)$ can be removed as implicitly follows from previous arguments. In particular, without adversarial contamination the technique of Catoni and Giulini \cite[Proposition 4.1]{catoni2017dimension} can be adapted to recover the desired rate (see also the discussion in \cite{giulini2017robust} and the bounds in \cite{giulini2018robust} where at least a $\log\log N$ term appears); their estimator involves some explicit Gaussian integrals in the parameter space and will depend on additional parameters of the distribution such as $\tr(\Sigma)$ and $\|\Sigma\|$.
    
    Shortly after the first version of this paper was made public, the authors were notified that Z. Fern\'andez-Rico and R. I. Oliveira achieved the same bound (only for $p = 4$) without the logarithmic factor in the setup without any contamination. Their result appeared later in the PhD thesis of Z. Rico \cite{FernandezRico2022} and in a preprint \cite{oliveira2022improved}.}. Secondly, since their estimator is based on Median-of-Means, it is currently not known how to significantly improve the dependence on $\eta$ beyond $\sqrt{\eta}$ in the bound of Theorem \ref{thm:informalmain}.  
\end{itemize}

The most involved step is to get the correct dependence on the corruption level $\eta$ when $p > 4$ in \eqref{eq:momeqv}. 
One of the technical results used in this paper is a dimension-free version of the classical Bai-Yin theorem \cite{baiyin1994}. In a nutshell, the result of Bai and Yin implies that if $X_1, \ldots, X_N$ are independent copies of a zero mean random vector $X$ in $\mathbb{R}^d$ with unit covariance and independent identically distributed coordinates having four bounded moments, that is $p = 4$, then 
\[
\left\|\frac{1}{N}\sum\limits_{i = 1}^NX_i\otimes X_i - I_d\right\| \to 2\sqrt{\frac{d}{N}} + \frac{d}{N}
\]
almost surely as $d, N \to \infty$ so that $d/N \to \beta \in (0, 1]$. Moreover, if the distribution only has $p < 4$ moments, no convergence at such a rate is possible. Our result is non-asymptotic and is somewhat stronger, since we do not require that the coordinates of $X_i$ are independent. Moreover, our focus is on the dimension-free bound. At the same time, we require that $p > 4$ and the constant in the bound depends on how well $p - 4$ is separated from zero. 
\begin{theorem}[A non-asymptotic, dimension-free Bai-Yin type theorem]
\label{thm:baiyin}
Assume that $Y$ is a zero mean random vector in $\mathbb{R}^d$ with covariance $\Sigma$ satisfying $L_p-L_2$ \emph{norm equivalence} with $p > 4$. Let $Y_1, \ldots, Y_N$ be a sample of independent copies of $Y$. Consider the truncated vectors $X_i = Y_i\ind{\|Y_i\| \le (N\tr(\Sigma)\|\Sigma\|)^{1/4}}$ for $i = 1, \ldots, N$. If $N  \ge c(p)\mathbf{r}(\Sigma)$, then it holds that
\[
\E\left\|\frac{1}{N}\sum\limits_{i = 1}^N \left(X_i \otimes X_i - \E X_i\otimes X_i\right)\right\| \le C(p)\|\Sigma\|\sqrt{\frac{\mathbf{r}(\Sigma)}{N}},
\]
where $c(\cdot)$ and $C(\cdot)$ are non-increasing and both satisfy $C(p), c(p) \to \infty$ as $p \to 4$.
\end{theorem}
A natural attempt to prove this bound would be to apply the matrix Bernstein inequality (see the survey \cite{tropp15} for more details on matrix concentration inequalities) as in \cite{mendelson2020robust}. However, in this case, we would get an additional multiplicative $\log \mathbf{r}(\Sigma)$-term that does not allow us to recover the optimal Bai-Yin rate of convergence in Theorem \ref{thm:informalmain} for $p > 4$. Alternatively one can try the generic chaining machinery. We refer to the monograph of Talagrand \cite{Talagrand2014} for a detailed exposition of this method. For example, the analysis based on the generic chaining is known to recover the optimal dependence on the effective dimension in the sub-Gaussian case \cite{koltchinskii2017operators}. The problem with using the generic chaining in our case, is that we only assume the existence of $p > 4$ moments, which does not allow  to control the increments of empirical processes with the required accuracy. The proof of Theorem \ref{thm:baiyin} is based on the combination of the (dimension dependent) approach developed by K. Tikhomirov \cite{tikhomirov2018sample} with the variational approach applied by the second-named author of this paper in \cite{zhivotovskiy2021dimension}.

\begin{remark}
Whether the bound of Theorem \ref{thm:baiyin} holds for $p = 4$ is a well-known open problem (see \cite{vershynin2012close} and the survey \cite{vershynin_2012}) even in the isotropic case where $\Sigma = I_d$. The analysis of Theorem \ref{thm:informalmain}  bypasses the bound of Theorem \ref{thm:baiyin} in the regime where $p = 4$. This makes our analysis different from the approach used in \cite{mendelson2020robust}.
\end{remark}

\begin{remark} Theorem \ref{thm:baiyin} can be used to remove the logarithm factor in the Median-of-Means estimator proposed in \cite{mendelson2020robust} for $p>4$.
\end{remark}
\paragraph{On the breakdown point of our estimator.} To explain informally, the breakdown point of an estimator is defined as the largest proportion of outliers in
the data for which the estimator gives a non-vacuous result \cite{hampel1971general}. The form of the bounds we are interested in is a high probability bound of the form $\|\widehat\Sigma - \Sigma\| \le \varepsilon \|\Sigma\|$, where $\widehat \Sigma$ is our estimator and $\varepsilon > 0$ is precision parameter. Since $\widehat \Sigma = 0$ satisfies this inequality with $\varepsilon = 1$, in Theorem \ref{thm:informalmain} we are focusing on $\eta \in [0, c]$, where $c$ is a small enough constant so that for large enough sample size we guarantee $\varepsilon < 1$. 

\paragraph{Practical considerations.} Our main focus is on the statistical properties of the estimator. In fact, we are aiming to achieve the best possible statistical performance and we left computational questions aside. While there is some algorithmic progress in the case of adversarial contamination (see the survey \cite{diakonikolas2019recent}), in a recent paper \cite{cherapanamjeri2020algorithms}, the authors provide an evidence that achieving the same for heavy-tailed distributions is computationally hard at least when Median-of-Means estimators are used. For some practical estimators in the context of heavy-tailed covariance estimation, we refer to \cite{ke2019user, hardle2021robustifying, minsker2022robust}.

\paragraph{Technical overview.} Our analysis connects three separate arguments in the literature. First, we prove a dimension-free version of the Bai-Yin theorem (Theorem \ref{thm:baiyin}) that allows us to eliminate unnecessary logarithmic factors in our bounds. The proof of this result consists of two parts: The analysis of the so-called \emph{peaky part} is based on the arguments of K. Tikhomirov  \cite{tikhomirov2018sample}, which in itself improves the line of research \cite{bourgain1996random, adamczak2010quantitative, mendelson2014singular, guedon2017interval}. The dimension-free analysis of the \emph{spread part} is based on the variational inequality techniques applied by the second-named author of this paper in \cite{zhivotovskiy2021dimension}. The latter approach traces back to the works of O. Catoni and co-authors on robust mean estimation \cite{audibert2011robust, catoni2016pac, giulini2018robust, catoni2017dimension}. Extending the results in \cite{zhivotovskiy2021dimension}, we prove Theorem \ref{thm:informalmain} in the case where $p = 4$. Proposition \ref{prop:estlargesteigenvalue} provides an optimal estimator for the largest eigenvalue of the covariance matrix and is based on Catoni's estimator \cite{catoni2012challenging, catoni2016pac}. Finally, the proof of Theorem \ref{thm:informalmain} in the regime $p > 4$ is based on combining Theorem \ref{thm:baiyin} with the second order version of the multivariate trimmed mean estimator of Lugosi and Mendelson \cite{lugosi2021robust}. Our lower bounds are based on reducing to corresponding lower bounds in the multivariate mean estimation setup.

\paragraph{Structure of the paper.} The rest of the paper is organized as follows. In Section \ref{sec:baiyin}, we present a proof of Theorem \ref{thm:baiyin}. In Section \ref{sec:pfour}, we provide a proof Theorem \ref{thm:informalmain} in the regime $p = 4$. In Section \ref{sec:opernorm}, we provide an auxiliary result on estimating the largest eigenvalue of the covariance matrix. Then, using Theorem \ref{thm:baiyin} in Section \ref{sec:pmorefour}, we give a proof Theorem \ref{thm:informalmain} in the regime $p > 4$.  We conclude with a detailed discussion of the lower bounds, showing the optimality of our results in Section \ref{sec:optimality}. 

\paragraph{Notation.} Throughout the proofs $C(p)$ and $c(p)$ will denote the constants depending only on $p$ and (possibly) on $\kappa(p)$, where $\kappa(\cdot)$ is given by \eqref{eq:momeqv}. The exact values of $C(p)$ and $c(p)$ may change from line to line. For an integer $N$, we set $[N] = \{1, \ldots, N\}$. Let $\ind{A}$ denote the indicator of an event $A$. The symbol $\mathbb{R}_{+}$ denotes the set of positive reals.  For any two functions (or random variables) $f, g$ defined in some common domain, the notation $f \lesssim g$ means that there is an absolute constant $c$ such that $f \le cg$ and $f\sim g$ means that $f\lesssim g$ and $g\lesssim f$. For $(x_i)_{i = 1}^m \in \mathbb{R}^m$ the sequence $(x_i^*)_{i = 1}^m \in \mathbb{R}^m$ is a non-increasing rearrangement of $(|x_i|)_{i = 1}^m$. For a set $I \subseteq [N]$ let $I^c = [N]\setminus I$. Let $\mathbb{S}_{+}^d$ denote the set of $d$ by $d$ positive-definite matrices. The symbol $\|\cdot\|$ denotes the operator norm of a matrix or the Euclidean norm of the vector depending on the context. The symbol $\|a\|_{0}$ corresponds to the number of non-zero components of the vector $a$. For a random variable $X$, let $\|X\|_{\infty}$ denote its essential supremum. Let $ \mathcal{KL}(\rho, \mu) = \int\log\left(\frac{d\rho}{d\mu}\right)d\rho$
denote the Kullback-Leibler divergence between a pair of measures $\rho$ and $\mu$.

\section{Proof of Theorem \ref{thm:baiyin}}
\label{sec:baiyin}
We start by proving the dimension-free version of the Bai-Yin theorem.
Fix $\lambda > 0$. In the notation of Theorem \ref{thm:baiyin}, let us write the following decomposition:

\begin{align*}
\sup_{v\in S^{d-1}} \left| \frac{1}{N}\sum\nolimits_{i=1}^N \langle X_i,v\rangle^2 - \E\langle X,v\rangle^2\right| &\le \underbrace{\sup_{v\in S^{d-1}}\frac{1}{N}\sum\nolimits_{i=1}^N \langle X_i,v\rangle^2 \ind{\lambda\langle X_i,v\rangle^2>1}}_{\textrm{Peaky part}}
\\
&\qquad+ 
\underbrace{\sup_{v\in S^{d-1}}\left|\frac{1}{N\lambda}\sum\nolimits_{i=1}^N \psi(\lambda\langle X_i,v\rangle^2) - \E\langle X,v\rangle^2 \right|}_{\textrm{Spread part}}~,
\end{align*}
where the terminology comes from \cite[Chapter 9.4]{Talagrand2014}. Our analysis will consist of two steps, where we first analyze the \emph{peaky part} and then analyze the \emph{spread part}.

\subsection{Dimension-free upper bound on the peaky part}
To analyze this term we follow the strategy of K. Tikhomirov in \cite{tikhomirov2018sample}, which in itself improves a line of research \cite{bourgain1996random,  adamczak2010quantitative, mendelson2014singular, Talagrand2014, guedon2017interval}. Although this part of our proof follows their steps, we need several modifications to avoid any explicit dependence on the dimension. Following \cite{tikhomirov2018sample} given two sets $C,I \subset [N]$ and  a positive integer $k$, we define the quantities  $f(k,C)$, $g(k,C,I)$ and $W_{v,i}$ as
\begin{equation*}
    f(k,C)=  \sup_{\substack{\|y\|_2=1 \\ \|y\|_0 \le k \\ \supp{y}\subset C }} \left\|\sum_{i=1}^N y_i X_i\right\|^2,
\end{equation*}
\begin{equation*}
g(k,C,I)= \sup_{\substack{\|y\|_2=1 \\ \|y\|_0 \le k}} \sup_{\substack{\|z\|_2=1 \\ \|z\|_0 \le k}} \left\langle \sum_{i\in I\cap C}y_iX_i,\sum_{j\in I^c \cap C}z_jX_j \right\rangle,
\end{equation*}
\begin{equation}
\label{eq:wvi}
W_{v,i}= \left\langle X_i,\sum_{j=1}^N v_jX_j \right\rangle.
\end{equation}
In particular, as in \cite[Equations 9.140 and 9.142]{Talagrand2014}, we have
\begin{equation}
\label{eq:relationfanda}
 \sup\limits_{v \in S^{d - 1}}\sup_{|I|\le k}\sum\limits_{i \in I}\langle v, X_i\rangle^2= \sup_{|I|\le k}\sup_{\sum_{i\in I}a_i^2 \le 1} \left\|\sum_{i\in I}a_iX_i\right\|^2 = f(k, [N]).
\end{equation}
We show below that, to control the peaky part, it is sufficient to upper bound the quantity $f(k, [N])$.
The role of $g(K, C, I)$ goes as follows: Using a standard decoupling argument (see the derivation in \cite[page 11]{tikhomirov2018sample}), we obtain that
\[
f(k,C) \le \max_{i\le N} \|X_i\|^2 + 2^{-N+2} \sum_{I\subset[N]}g(k,C,I).
\]
Since $\max_{i\le N}\|X_i\|^2 \le (N\tr(\Sigma)\|\Sigma\|)^{1/2}$ by assumption, we limit our attention on bounding $g(k,C,I)$. We use the following auxiliary bound. Its application should be compared with the application of Lemma 2.6 in \cite{tikhomirov2018sample}. Due to the dimension-free nature of our bound, we do not capture the best possible dependence on $p$ in the bound below and only exploit the case of $p = 4$. Despite that, this still allows us to prove the desired result.
\begin{lemma}
\label{lem:normofthevec}
Let $Y \in \mathbb{R}^d$ be a zero mean random vector with covariance $\Sigma$ satisfying the $L_p-L_2$ norm equivalence for $p>4$ with $\kappa(\cdot)$ in \eqref{eq:momeqv}. 
Denote $X = Y\ind{\|Y\| \le R}$ where $R > 0$ and let $\kappa = \kappa(4)$. It holds that
\[
\mathbb{P}(\|X\|_2\ge t) \le \kappa^{4}\frac{\tr(\Sigma)^2}{t^4}.
\]
\end{lemma}
\begin{proof} We have
\begin{equation*}
\E\|X\|_2^4 = \E\left(\sum\nolimits_{i=1}^d \langle X_i,e_i\rangle^2\right)^2 =\E \sum\nolimits_{i=1}^d \langle X,e_i\rangle^4 + \E\sum\nolimits_{k\neq l} \langle X,e_k\rangle^2\langle X,e_l\rangle^2.
\end{equation*}
We use that, for all $i\in [d]$, 
\[
(\E\langle X,e_i\rangle^4)^{1/4} \le (\E\langle Y,e_i\rangle^4)^{1/4}\le \kappa (\E\langle Y,e_i\rangle^2)^{1/2} = \kappa\Sigma_{ii}^{1/2}.
\]
Now, by the above display and H\"{o}lder's inequality we have
\begin{align*}
\kappa^4 \sum\nolimits_{i=1}^d (\E \langle X,e_i\rangle^2)^2 + \kappa^{4}\sum\nolimits_{k \neq l} \E\langle X,e_k\rangle^2 \E\langle X,e_l\rangle^2 &= \kappa^4\left(\sum\nolimits_{j=1}^d \Sigma_{jj}^2 + \sum\nolimits_{k \neq l} \Sigma_{kk} \Sigma_{ll}\right) 
\\
&= \kappa^4 \tr(\Sigma)^2.
\end{align*}
The Markov's inequality concludes the proof.
\end{proof}

Following \cite{tikhomirov2018sample}, the next step is to split the vectors $X_1, \ldots, X_N$ is several groups so that the vectors in each group are almost orthogonal.  Fix $H > 0$. Consider the random graph $\mathcal G_{H} = ([N], E)$, where the set of edges is
\[
E = \left\{(i, j): 1 \le i < j \le N, \quad |\langle X_i, X_j\rangle| > H\max\limits_{k \le N}\|X_k\|\right\}.
\]
If we color this graph so that no two adjacent vertices share the same color, we have that the vertices of the same color correspond to a pair of vectors having a small inner product. Let $\chi(\mathcal G_{H})$ denote the chromatic number of $\mathcal G_{H}$. We now induce the sets $\{\mathcal{C}_m^H\}_{m = 1}^N$ performing the partition of $[N]$ such that $\mathcal{C}_m^H = \emptyset$ for $m > \chi(\mathcal G_{H})$ and each $\mathcal{C}_m^H$ contains the vertices of the same color. Following the proof of the dimension-dependent result, we provide a dimension-free bound on the chromatic number of this random graph. 

\begin{proposition}
\label{prop:coloringprop}
Let $Y \in \mathbb{R}^d$ be a zero mean random vector with covariance $\Sigma$ satisfying the $L_p-L_2$ norm equivalence for $p>4$ with $\kappa(\cdot)$ as in \eqref{eq:momeqv}. Define, for an arbitrary positive $R$, $X = Y\ind{\|Y\| \le R}$ and let $\kappa = \kappa(4)$. Then, for any $H>0$ and any integer $m>1$, with probability at least $1-(\kappa(p)^p\|\Sigma\|^{p/2}NH^{-p})^{m-1} N\kappa(p)^4 \tr(\Sigma)^2 H^{-4}$,
\[
\chi(\mathcal{G}_H) \le m.
\]
\end{proposition}

\begin{proof}
We consider a greedy coloring process. Let $Y(1), Y(2), \ldots$ be an auxiliary random process such that $Y(1)=1$ and
\begin{equation*}
Y(i)= \min \{ r\in \mathbb{N}: \forall j<i \ \text{with}\  Y(j)=r \ \text{we have}\ |\langle X_i,X_j\rangle|\le H\|X_j\| \}.
\end{equation*}
Note that, by definition, $\chi(\mathcal{G}_H) \le \max_{i\in [N]}Y(i)$. Next, for each $i>1$ and $m\ge 1$, we have
\begin{align*}
\mathbb{P}(Y(i)=m+1) &\le \mathbb{P}(\text{There is}\ l\le i-1 \  \text{such that} \ |\langle X_i,X_j\rangle|> H\|X_j\| \ \text{and}\ Y(l)=m)\\
& \le \sum_{l=1}^{i-1} \mathbb{P}(|\langle X_i,X_j\rangle|> H\|X_j\|\ \text{and}\ Y(l)=m)\\
& = \sum_{l=1}^{i-1} \mathbb{P}(|\langle X_i,X_j\rangle|> H\|X_j\|\ |\  Y(l)=m)\mathbb{P}(Y(l)=m)\\
& \le \kappa(p)^p\frac{\|\Sigma\|^{p/2}}{H^p}\sum_{l=1}^{i-1} \mathbb{P}(Y(l)=m)
\\
& \le \kappa(p)^p\frac{\|\Sigma\|^{p/2}}{H^p} \mathbb{E}|\{j\le N: Y(j) = m\}|,
\end{align*}
where we used that since $X_i$ is independent of $Y(l)$, and by Markov's inequality
\begin{align*}
\mathbb{P}(|\langle X_i,X_j\rangle|> H\|X_j\|\ |\  Y(l)=m) &= \mathbb{P}(|\langle X_i,X_j\rangle|> H\|X_j\| \ \textrm{and} \ \|X_j\| \neq 0\ |\  Y(l)=m) 
\\
& \le \frac{\kappa(p)^{p}\|\Sigma\|^{p/2}}{H^{p}}.
\end{align*}
We obtain the following recursion,
\begin{equation*}
\mathbb{E}|\{j\le N: Y(j) = m+1\}| \le \kappa(p)^p\|\Sigma\|^{p/2}N\frac{1}{H^p} \mathbb{E}|\{j\le N: Y(j) = m\}|
\end{equation*}
Now we apply Lemma \ref{lem:normofthevec} and use the monotonicity of $\kappa(\cdot)$ to obtain,
\begin{equation*}
\mathbb{E}|\{j\le N: Y(j) = 2\}| \le N \mathbb{P}(\|X\|>H) \le N\kappa(p)^4 \tr(\Sigma)^2 H^{-4}.
\end{equation*}
Combining the estimates above, we get 
\[
\mathbb{E}|\{j\le N: Y(j) = m+1\}|\le (\kappa(p)^p\|\Sigma\|^{p/2}NH^{-p})^{m-1} N\kappa(p)^4 \tr(\Sigma)^2 H^{-4}.
\]
Finally, we obtain
\begin{align*}
\mathbb{P}(\chi(\mathcal{G}_{H}) \ge m+1) &\le \mathbb{P}(\exists j\le N: \  Y(j)=m+1)
\\
& \le \mathbb{E}|\{j\le N: Y(j) = m+1\}|
\\
&\le (\kappa(p)^p\|\Sigma\|^{p/2}NH^{-p})^{m-1} N\kappa(p)^4 \tr(\Sigma)^2 H^{-4}.
\end{align*}
\end{proof}
We need the following result. It can be seen as one of the main ingredients of the proof and follows from the so-called \emph{Sparsifying} lemma \cite[Lemma 4.1]{tikhomirov2018sample}. Our key observation here is that this lemma does not involve explicitly the dimension $d$. Recall that for $(x_i)_{i = 1}^m \in \mathbb{R}^m$ the sequence $(x_i^*)_{i = 1}^m \in \mathbb{R}^m$ is a non-increasing rearrangement of $(|x_i|)_{i = 1}^m$ and $v$ is an $s$-sparse vector if $\|v\|_0 \le s$. We remark that the statement below is deterministic and holds for any realization of $X_1, \ldots, X_N$.

\begin{proposition}[Proposition 4.4 in \cite{tikhomirov2018sample}]
\label{prop:sparsification}
There exists an absolute constant $C>0$ such that the following holds.
Fix $I \subseteq [N]$ and let $\gamma \in (0, 1/3)$ with $k \ge 24/\gamma^2$ and $N \ge 128C\gamma^{-2}k$. Set $t = \lfloor \log_2 \frac{\gamma^2 k}{24}\rfloor$ and $k_j = \lfloor \frac{k}{2^j}\rfloor$ for $0 \le j \le t$. There are subsets $\mathcal N_j, \mathcal N^{\prime}_j$ for $0 \le j \le t - 1$ supported on $I$ and $I^c$ respectively and consisting of unit $\gamma k_j$-sparse vectors such that 
\[
\max\left\{\mathcal N_j, \mathcal N^{\prime}_j \right\} \le \left(\frac{CN}{\gamma k_j}\right)^{2\gamma k_j},
\]
and for any $\mathcal C \subseteq [N]$,
\[
g(k, \mathcal C, I) \le C\gamma^{-2}\left(\log k \max\limits_{i \neq j \in \mathcal C}|\langle X_i, X_j\rangle| + \sum\limits_{j = 0}^{t - 1}\sqrt{k_j}\left(\sup\limits_{u \in \mathcal N_j}\left(A_{u}\right)^*_{\left\lfloor \frac{k_{j + 1}}{16}\right\rfloor} + \sup\limits_{v \in \mathcal N^{\prime}_j}\left(B_{v}\right)^*_{\left\lfloor \frac{k_{j + 1}}{16}\right\rfloor}\right)\right),
\]
where the random vectors $A_u \in \mathbb{R}^{|I^c|}$ and $B_v \in \mathbb{R}^{|I|}$ are given by $A_u = (|W_{u, i}|)_{i \in I^c}$ and $B_v = (|W_{v, i}|)_{i \in I}$ and $W_{u, i}, W_{v, i}$ are given by \eqref{eq:wvi}.

\end{proposition}

The next result is a dimension-free analog of Proposition 5.1 in \cite{tikhomirov2018sample} in the regime $p > 4$.
\begin{proposition}
\label{prop:gfuncupperbound}
There exists an absolute constant $C>0$ such that if $I\subset [N]$ with $|I|\le s$ is fixed and $\log \frac{N}{s}\ge C$, then simultaneously for all  $\mathcal{C}\subset [N]$, with probability at least $1-1/N^3$,
\begin{equation*}
g(s,\mathcal{C},I) \le C\left(\log^2\frac{N}{s}\log s\max_{i\neq j \in \mathcal{C}}|\langle X_i,X_j\rangle| + p\kappa(p) \|\Sigma\|^{1/2}\log^2\frac{N}{s}\sqrt{s}\left( \frac{N}{s}\right)^{1/p} \sqrt{f(s, [N])}\right) .
\end{equation*}
\end{proposition}
\begin{proof}
We provide the required changes of the original proof.
First, we notice that if $s<C \log^2\frac{N}{s}$, then 
\begin{equation*}
g(s,\mathcal{C},I)\le s\max_{i\neq j \in \mathcal{C}}|\langle X_i,X_j\rangle|<C \log^2\frac{N}{s}\max_{i\neq j \in \mathcal{C}}|\langle X_i,X_j\rangle|,
\end{equation*}
and the claim trivially follows.

For the rest of the proof we assume $s\ge C \log^2\frac{N}{s}$. We define $\gamma=\frac{1}{\log(N/s)}$, $t=\lfloor\log_2\frac{\gamma^2 s}{C}\rfloor$ and $k_j=\lfloor\frac{s}{2^j}\rfloor$. We choose $C>24$ to be able to apply Proposition \ref{prop:sparsification} with this choice of $\gamma, t$ and $k_j$. For a fixed $0 \le j < t$ consider $u \in \mathcal N_j$, where $\mathcal N_j$ is given by Proposition \ref{prop:sparsification}. By its definition $u$ is supported on $I$. Therefore, for any $l \in I^{c}$ recalling \eqref{eq:wvi}, we have $W_{u,l}=\langle X_l,\sum_{i\in I}u_iX_i\rangle$. Observe that conditionally on the realization of $X_i, i \in I$, we have that $W_{u,l}$ are independent random variables for all $l \in I^c$. Therefore, by the norm equivalence assumption \eqref{eq:momeqv}, we have
\[
\E[|W_{u,l}|^p|X_i, i\in I] \le \kappa(p)^p \|\Sigma\|^{p/2} \left\|\sum\nolimits_{i\in I}u_iX_i\right\|^p \le \kappa(p)^p \|\Sigma\|^{p/2}f(s, [N])^{p/2}.
\]
For the same vector $u$ and the same set $I$, consider the random vector 
\[
A_u = (|W_{u, l}|)_{l \in I^c}.
\]
We apply the standard bound \cite[Lemma 2.5]{tikhomirov2018sample} to control the coordinates of $A_u$ with
\[
\tau_j=(32e)^{1/p}\|\Sigma\|^{1/2}\kappa(p)\sqrt{f(s, [N])}\left(\frac{N}{k_{j+1}}\right)^{p^{-1}\left(1+256\gamma\right)}
\]
to obtain
\begin{equation*}
\begin{split}
\mathbb{P}((A_u)^*_{\lfloor k_{j + 1}/16\rfloor} \ge \tau_j)&\le \left(\frac{e\kappa^p\|\Sigma\|^{p/2}f(s, [N])^{p/2}N}{\tau_j^p\lfloor k_{j+1}/16\rfloor}\right)^{\lfloor k_{j+1}/16\rfloor}\\
& \le \left(\frac{k_{j+1}}{N}\right)^{4\gamma k_j}.
\end{split}
\end{equation*}
We union bound over the net $\mathcal{N}_j$, whose size is bounded in Proposition \ref{prop:sparsification}, to get for some absolute constant $C_1 > 0$,
\begin{equation*}
\mathbb{P}\left(\sup\limits_{u \in \mathcal N_j}(A_u)^*_{\lfloor k_{j + 1}/16\rfloor}\ge \tau_j\right)\le \left(\frac{k_{j+1}}{N}\right)^{4\gamma k_j}|\mathcal{N}_j| \le \left(\frac{C_1k_j}{\gamma N}\right)^{2\gamma k_j}\le  \left(\frac{k_j}{N}\right)^{\gamma k_j} \le \left(\frac{k_t}{N}\right)^{\gamma k_t} \le \frac{1}{N^4},
\end{equation*}
where we assumed $N\ge C_1^2\gamma^{-2}k_j$. Repeating the same arguments for $\mathcal{N}^{\prime}_j$ and summing over all $j$ in Proposition \ref{prop:sparsification} we obtain, with probability at least $1 - 1/N^3$,
\[
g(s, \mathcal C, I) \lesssim \gamma^{-2}\left(\log s \max_{i\neq j \in \mathcal{C}}|\langle X_i,X_j\rangle| + \sum\nolimits_{j = 0}^{t - 1}\tau_j\sqrt{k_j}\right).
\]
The sum $\sum\nolimits_{j = 0}^{t - 1}\tau_j\sqrt{k_j}$ can be upper bounded exactly the same was as in the the end of the proof of Proposition 5.1 in \cite{tikhomirov2018sample}. Using $p > 4$, we conclude the proof.
\end{proof}
The next result applies Proposition \ref{prop:gfuncupperbound} to upper bound $f(s,\mathcal{C}_m^H)$, where $\mathcal{C}_m^H$ is a class of coloring of the sample. 
\begin{lemma}
\label{lem:individualcolor}
Let $s\le N$ be fixed and assume that $\log\frac{N}{s}\ge C_1$, where $C_1>0$ is a large enough absolute constant. Let $\mathcal{C}_m^{H}$ be the class from the coloring of the sample $X_1,\ldots,X_N$ with threshold $H>0$. Then there exists an absolute constant $C>0$ such that, with probability at least $1-\frac{1}{N^2}$,
\[
f(s,\mathcal{C}_m^H) \le \max_{i\le N}\|X_i\|^2 + C\log^2\frac{N}{s}\left(H\log s \max_{i\le N}\|X_i\| + p\kappa\|\Sigma\|^{1/2}\sqrt{s}\left( \frac{N}{s}\right)^{1/p}\sqrt{f(s,[N])}\right).
\]
\end{lemma}
\begin{proof}
The proof follows exactly the same steps (an application of Proposition \ref{prop:gfuncupperbound} in our case) as the proof of Lemma 5.2 in \cite{tikhomirov2018sample} by replacing $d$ (their notation for the dimension is $n$) by $s$.  
\end{proof}
Our final step is to upper bound the desired function $f(s,[N])$ using the upper bound on $f(s,\mathcal{C}_m^H)$ for different coloring classes. Our new observation is that since we are only interested in the regime $p > 4$, we can worsen the dependence on some logarithmic factors appearing in the corresponding result in \cite{tikhomirov2018sample}. This allows us to build our analysis on a somewhat weaker Lemma \ref{lem:normofthevec} (compare it with a dimension-dependent result in \cite[Lemma 2.6]{tikhomirov2018sample}).  
\begin{lemma}
\label{lem:upboundcolor}
Let $s,N$ satisfy $\log\frac{N}{s}\ge C_1$, where $C_1>0$ is a large enough absolute constant. Then there exits $\chi = \chi(p)$ that depends only on $p$ such that
\begin{equation*}
f(s,[N]) \lesssim \chi^2 \max_{i\le N}\|X_i\|^2+ \kappa^2\|\Sigma\|s\left(\frac{N}{s}\right)^{4/(4+p)}\log^4\frac{N}{s} + \chi^2p^2\kappa(p)^2\|\Sigma\|s\log^4\frac{N}{s}\left( \frac{N}{s}\right)^{2/p},
\end{equation*}
 with probability at least $1-N^{(4 - p)(\chi - 2)/(p + 4)} s^{-(p^2(\chi- 1) + 4p)/(2p + 8)} (\log s)^{p(\chi - 1) + 4} \mathbf{r}^2(\Sigma)-\chi/N^2$.
 \end{lemma}
\begin{proof}
We define $\chi(\mathcal{G}_H)$ to be the chromatic number of the random graph induced by the partition of the sample $X_1,\ldots,X_N$ into the classes $\mathcal{C}_m^{H}$. We fix $H =\kappa\|\Sigma\|^{1/2}(\frac{N}{s})^{2/(4+p)} \frac{\sqrt{s}}{\log s}$, where $\kappa = \kappa(p)$ and apply Lemma \ref{lem:individualcolor} together with union bound to obtain that, with probability at least $1-\chi/N^2$,
\begin{align*}
\frac{1}{\chi}\sum_{m=1}^{\chi} f(s,\mathcal{C}_m^{H}) &\le \max_{i\le N}\|X_i\|^2
\\
&\quad+ C\log^2\frac{N}{s}\left(H\log s \max_{i\le N}\|X_i\| + p\kappa\|\Sigma\|^{1/2}\sqrt{s}\left( \frac{N}{s}\right)^{1/p} \sqrt{f(s,[N])}\right).
\end{align*}
We now apply the dimension-free version of the coloring Proposition \ref{prop:coloringprop} to obtain that the chromatic number $\chi(\mathcal{G}_H)$ is at most $\chi$, with probability at least \[
1-(\kappa^p\|\Sigma\|^{p/2}NH^{-p})^{\chi-1} N\kappa^4 \tr(\Sigma)^2 H^{-4}.
\]
Now, we estimate the latter probability. We have
\begin{align*}
&(\kappa^p \|\Sigma\|^{p/2} N H^{-p})^{\chi-1}N\kappa^4\tr(\Sigma)^2H^{-4} 
\\
&\qquad= (N^{1-2p/(p+4)}s^{2p/(p+4)-p/2})^{\chi-1} N^{1-8/(4+p)}s^{8/(4+p) - 2} (\log s)^{p(\chi - 1) + 4} \mathbf{r}^2(\Sigma)\\
&\qquad=N^{(4 - p)(\chi - 2)/(p + 4)} s^{-(p^2(\chi- 1) + 4p)/(2p + 8)} (\log s)^{p(\chi - 1) + 4}\mathbf{r}^2(\Sigma).
\end{align*}
Following exactly the lines of the proof of Proposition 5.3 in \cite{tikhomirov2018sample} we obtain
\begin{align*}
\chi^{-1}(\mathcal{G_H})f(s,[N]) &\le \chi^{-1}(\mathcal{G_H})\sum_{m=1}^{\chi(\mathcal{G}_H)} f(s,\mathcal{C}_m^{H}) \\
&\le\max_{i\le N}\|X_i\|^2+ C\kappa\|\Sigma\|^{1/2}\left(\frac{N}{s}\right)^{2/(4+p)} \sqrt{s}\log^2\frac{N}{s} \max_{i\le N}\|X_i\| 
\\
&\qquad+ Cp\kappa\|\Sigma\|\log^2\frac{N}{s}\sqrt{s}\left( \frac{N}{s}\right)^{1/p} \sqrt{f(s,[N])},
\end{align*}
with probability at least $1-N^{(4 - p)(\chi - 2)/(p + 4)} s^{-(p^2(\chi- 1) + 4p})/(2p + 8)(\log s)^{p(\chi - 1) + 4} \mathbf{r}^2(\Sigma)$. Now we solve the inequality above by applying the inequality $2ab \le \gamma a^2 + \frac{b^2}{\gamma}$ for $a, b, \gamma \ge 0$ twice and solving with respect to $f(s,[N])$.
\end{proof}

The following bound is the main result of this section.

\begin{theorem}
\label{thm:rearrengmentthm}
Assume that for some large enough absolute constant $c > 0$ it holds that $N \ge c\mathbf{r(\Sigma)}$.
For large enough sample size $N$, simultaneously for all integers $s$ satisfying $\mathbf{r}(\Sigma) \le s \le N/c$, with probability at least $1 - \frac{c(p)}{N}$, it holds that
\begin{equation}
\label{eq:boundonf}
f(s,[N]) \le C(p) \left( \max_{i\le N}\|X_i\|^2 + \|\Sigma\|s\left(\frac{N}{s}\right)^{4/(4+p)}\log^4\frac{N}{s} \right),
\end{equation}
where $C(p)$ and $c(p)$ depend only on $p$ and $\kappa(p)$.
\end{theorem}

\begin{proof}
The proof is based on the application of Lemma \ref{lem:upboundcolor}.
Let $s \ge \mathbf{r}(\Sigma)$ and fix $\chi = \max(10,\frac{4p}{p-4})$. Consider two cases:
\begin{itemize}
\item If $\mathbf{r}(\Sigma) \le 100$, we have
\[
N^{(4 - p)(\chi - 2)/(p + 4)} s^{(p^2(\chi- 1) +4p)/(2p + 8)}(\log s)^{p(\chi - 1) + 4} \mathbf{r}^2(\Sigma) \le c(p)N^{-2}.
\]
\item Otherwise, if $\mathbf{r}(\Sigma) > 100$, then we have $s \ge 100$ and $\log s \le s^{1/3}$. Therefore,
\[
N^{(4 - p)(\chi - 2)/(p + 4)} s^{-(p^2(\chi- 1) + 4p)/(2p + 8) + (p(\chi - 1) + 4)/3} \mathbf{r}^2(\Sigma)  \le s^{-2}\mathbf{r}^2(\Sigma)N^{-2} \le N^{-2}.
\]
\end{itemize}
Therefore, for a fixed integer $s$ satisfying $\mathbf{r}(\Sigma) \le s \le N/c$, we have that $\log \frac{N}{s}$ is sufficiently large so that, by Lemma \ref{lem:upboundcolor} and the above calculations, with probability at least $1 - \frac{c(p)}{N^2}$,
\[
f(s,[N]) \le C(p) \left( \max_{i\le N}\|X_i\|^2 + \|\Sigma\|s\left(\frac{N}{s}\right)^{4/(4+p)}\log^4\frac{N}{s} \right).
\]
By taking the union bound with respect to the value of $s$ we conclude the proof.
\end{proof}

\subsection{Analysis of the spread part}
The main tool of this section is the following lemma sometimes called the PAC-Bayesian inequality (see \cite[Proposition 2.1]{catoni2017dimension} or \cite{appert2021new} for a detailed proof taking care of measurability questions). In our proofs, we adapt several computations appearing in \cite{zhivotovskiy2021dimension}. We also refer to \cite{zhivotovskiy2021dimension} for a detailed exposition of related results.
\begin{lemma}
\label{lem:pacbayes}
Assume that $X_i$, $i = 1, \ldots, N$ are i.i.d. random variables defined on some measurable space. Assume also that $\Theta$ (called the parameter space) is a subset of $\mathbb{R}^d$ for some $d \ge 1$. Let $\mu$ be a distribution (called prior) on $\Theta$ and let $\rho$ be any distribution (called posterior) on $\Theta$ such that $\rho \ll \mu$. Then, simultaneously for any such $\rho$ we have, with probability at least $1 - \exp(-t)$,
\[
\frac{1}{N}\sum\limits_{i = 1}^N\E_{\rho}f(X_i, \theta) \le \E_{\rho}\log(\E_X\exp(f(X, \theta))) + \frac{\mathcal{KL}(\rho, \mu) + t}{N}.
\]
Here $\theta$ is distributed according to $\rho$. Moreover, 
\[
\E\sup\limits_{\rho}\left(\sum\limits_{i = 1}^N\E_{\rho}f(X_i, \theta) - N\E_{\rho}\log(\E_X\exp(f(X, \theta))) - \mathcal{KL}(\rho, \mu)\right) \le 0.
\]
\end{lemma}

Our analysis will also exploit the following elementary relation.

\begin{lemma}[Lemma 4 in \cite{zhivotovskiy2021dimension}]
\label{lem:almostconvex}
Let the truncation function $\psi$ be given by \eqref{eq:truncfunction} and let $Z$ be a square integrable random variable. We have
\[
\psi(\E Z) \le \E\log(1 + Z + Z^2) + \min\{1, \E Z^2/6\}.
\]
Moreover for any $a > 0$, it holds that
\[
 \E\log(1 + Z + Z^2) + a\E\min\{1, Z^2/6\} \le \E\log\left(1 + Z + \left(1 + \frac{(7 + \sqrt{6})(\exp(a) - 1)}{6}\right) Z^2\right).
\]
\end{lemma}

The main result of this section is the following.

\begin{proposition}
\label{prop:trimmedprocess}
Assume that $Y$ is a zero mean random vector with covariance $\Sigma$ satisfying the $L_p-L_2$ \emph{norm equivalence} with $p \ge 4$. Let $Y_1, \ldots, Y_N$ be a sample of independent copies of $Y$. Consider the truncated vectors $X_i = Y_i\ind{\|Y_i\| \le R}$ for $i = 1, \ldots, N$ and some $R > 0$. For a fixed truncation level $\lambda > 0$, it holds that
\[
\E\sup_{v\in S^{d-1}}\left|\frac{1}{N\lambda}\sum_{i=1}^N \psi(\lambda\langle X_i,v\rangle^2) - \E\langle X,v\rangle^2 \right|
\lesssim \frac{\mathbf{r}(\Sigma)}{\lambda N} + \lambda\kappa^{4}\|\Sigma\|^2,
\]
where $\kappa = \kappa(4)$. In particular, when $\lambda = \frac{1}{\kappa^2\|\Sigma\|} \sqrt{\frac{\mathbf{r}(\Sigma)}{N}}$, we have
\[
\E\sup_{v\in S^{d-1}}\left|\frac{1}{N\lambda}\sum_{i=1}^N \psi(\lambda\langle X_i,v\rangle^2) - \E\langle X,v\rangle^2 \right|
\lesssim \kappa^2\|\Sigma\|\sqrt{\frac{\mathbf{r}(\Sigma)}{N}}.
\]
\end{proposition}

\begin{proof}
Our aim is to choose the distributions $\mu$ and $\rho$ in Lemma \ref{lem:pacbayes}. Let 
$
\Theta = (\mathbb{R}^d)^2.
$
We choose $\mu$ to be a product of two zero mean multivariate Gaussians with mean zero and covariance $\beta^{-1}I_d$. For $v \in S^{d - 1}$, let $\rho_{v}$ be a product of two multivariate Gaussian distribution with mean $v$ and covariance $\beta^{-1}I_d$. Because of this, if $(\theta, \nu)$ is distributed according to $\rho_{v}$, we have $\E_{\rho_{v}}(\theta, \nu) = (v, v)$. By the additivity of $\mathcal{KL}$-divergence for product measures and the standard formula, we have
\[
\mathcal{KL}(\rho_{v}, \mu) = \beta.
\]
By the first part of Lemma \ref{lem:almostconvex} we have, 
\begin{align}
\psi\left(\lambda\langle X,  v\rangle^2\right) &=\psi\left(\lambda\E_{\rho_{v}}\langle X, \theta\rangle\langle X, \nu\rangle\right) \nonumber
\\
&\le \E_{\rho_{v}}\log\left(1 + \lambda\langle X, \theta\rangle\langle X, \nu\rangle  + \lambda^2\left(\langle X, \theta\rangle\langle X, \nu\rangle\right)^2 \right) \nonumber
\\
&\qquad+ \min\{1, \lambda^{2}\E_{\rho_v}\left(\langle X, \theta\rangle\langle X, \nu\rangle\right)^2/6\}.
\label{eq:twoparts}
\end{align}
Observe that 
\begin{equation}
\label{eq:forthnorm}
\E_{\rho_v}\left(\langle X, \theta\rangle\langle X, \nu\rangle\right)^2 = (\langle X, v\rangle^2 + \beta^{-1}\|X\|^2)^2 \le 2\langle X, v\rangle^4 + 2\beta^{-2}\|X\|^4,
\end{equation}
and we can write
\[
\min\{1, \lambda^{2}\E_{\rho_v}\left(\langle X, \theta\rangle\langle X, \nu\rangle\right)^2/6\} \le \min\{1, 2\lambda^2\langle X, v\rangle^4/6\} + \min\{1, 2\lambda^2\beta^{-2}\|X\|^4/6\}.
\]
Conditionally on $X$ the distribution of $\langle X, \theta\rangle$ is Gaussian with mean $\langle X, v\rangle$. Since it is symmetric, we have that $\Pr_{\rho_v}\left((\langle X, \theta\rangle)^2(\langle X, \nu\rangle)^2 \ge \langle X, v\rangle^4\right) \ge \frac{1}{4}$ and this holds trivially when $X = 0$. Therefore,
\[
\min\{1, 2\lambda^2\langle X, v\rangle^4/6\} \le 8\E_{\rho_v}\min\{1, \lambda^{2}\left(\langle X, \theta\rangle\langle X, \nu\rangle\right)^2/6\}.
\]
By the second part of Lemma \ref{lem:almostconvex}, we have for some absolute constant $c > 0$,
\begin{align*}
&\E_{\rho_{v}}\log\left(1 + \lambda\langle X, \theta\rangle\langle X, \nu\rangle  + \lambda^2\left(\langle X, \theta\rangle\langle X, \nu\rangle\right)^2 \right) + 8\E_{\rho_v}\min\{1, \lambda^{2}\left(\langle X, \theta\rangle\langle X, \nu\rangle\right)^2/6\}
\\
&\qquad\le \E_{\rho_{v}}\log\left(1 + \lambda\langle X, \theta\rangle\langle X, \nu\rangle  + c\lambda^2\left(\langle X, \theta\rangle\langle X, \nu\rangle\right)^2 \right).
\end{align*} Following the proof of Lemma \ref{lem:normofthevec} we get
\begin{equation}
    \label{eq:normforth}
    \E\|X\|^4 \le  \E\|Y\|^4 \le \kappa^4(\tr(\Sigma))^2 \quad \text{and} \quad \E\langle X, v\rangle^2 \le \E\langle Y, v\rangle^2 \le \|\Sigma\|,
\end{equation}
where $v \in S^{d - 1}$.
Using $\log(1 + y) \le y$ for $y \ge - 1$ and Fubini's theorem,  we have
\begin{align*}
&\E_{\rho_{v}}\log\E\left(1 + \lambda\langle X, \theta\rangle\langle X, \nu\rangle  + c\lambda^2\left(\langle X, \theta\rangle\langle X, \nu\rangle\right)^2 \right)
\\
&\le \lambda \E_{\rho_{v}}\E\langle X, \theta\rangle\langle X, \nu\rangle + c\lambda^2\E\E_{\rho_v}\langle X, \theta\rangle^2\langle X, \nu\rangle^2 
\\
&\le \lambda \E\langle X, v\rangle^2 + 2c\lambda^2(\E\langle X, v\rangle^4 + \E\beta^{-2}\|X\|^4) \quad \text{(by \eqref{eq:forthnorm})}
\\
&\le \lambda \E\langle X, v\rangle^2+ 2c\lambda^2\kappa^{4}((\E\langle X, v\rangle^2)^2 + \beta^{-2}(\tr(\Sigma))^2) \quad \text{(by \eqref{eq:normforth})}
\\
&\le \lambda \E\langle X, v\rangle^2 + 2c\lambda^2\kappa^{4}(\|\Sigma\|^2 + \beta^{-2}(\tr(\Sigma))^2).
\end{align*}
We plug 
\[
f(X, \theta, \nu) = \log\left(1 + \lambda\langle X, \theta\rangle\langle X, \nu\rangle  + c\lambda^2\left(\langle X, \theta\rangle\langle X, \nu\rangle\right)^2 \right)
\]
in the second part of Lemma \ref{lem:pacbayes}. Choosing $\beta = \mathbf{r}(\Sigma)$ and dividing both sides by $N$, we have
\begin{equation}
\label{eq:expectedsuptruncated}
\E\sup\limits_{v \in S^{d - 1} \cup \{0\}}\left(\frac{1}{N}\sum\limits_{i = 1}^N \E_{\rho_v}f\left(X_i, \theta, \nu\right) - \lambda\E\langle X, v\rangle^2\right) \le \frac{\mathbf{r}(\Sigma)}{ N} + 4c\lambda^2\kappa^{4}\|\Sigma\|^2,
\end{equation}
where we added the $0$ vector by considering $\mu$ as a posterior distribution and observing that $0 = \mathcal{KL}(\mu, \mu) \le \beta$. By adding the $0$ vector we guarantee that the supremum in \eqref{eq:expectedsuptruncated} is always non-negative. Adding the term
\[
\frac{1}{N}\E \sum\nolimits_{i = 1}^N\min\{1, 2\lambda^2(\mathbf{r}(\Sigma))^{-2}\|X_i\|^4/6\} \le \lambda^2 \frac{\kappa^4\|\Sigma\|^2}{3},
\]
to the inequality \eqref{eq:expectedsuptruncated} and using the derivations in the beginning of the proof, we have
\[
\E\sup\limits_{v \in S^{d - 1} \cup \{0\}}\left(\frac{1}{N}\sum\limits_{i = 1}^N\psi(\lambda\langle X,  v\rangle^2) - \lambda\E\langle X, v\rangle^2\right) \le \frac{\mathbf{r}(\Sigma)}{N} + \left(4c + \frac{1}{3}\right)\lambda^2\kappa^{4}\|\Sigma\|^2.
\]
The same argument works for $-\lambda$ instead. This leads to
\[
\E\sup\limits_{v \in S^{d - 1}}\left|\frac{1}{\lambda N}\sum\limits_{i = 1}^N\psi(\lambda\langle X,  v\rangle^2) - \E\langle X, v\rangle^2 \right| \le 2\frac{\mathbf{r}(\Sigma)}{\lambda N} + 2\left(4c + \frac{1}{3}\right)\lambda\kappa^{4}\|\Sigma\|^2.
\]
Our choice of $\lambda$ concludes the proof.
\end{proof} 
\subsection{Combining peaky and spread parts}
We conclude the proof of Theorem \ref{thm:baiyin}. Using our decomposition we have
\begin{align*}
\E\sup_{v\in S^{d-1}} \left| \frac{1}{N}\sum_{i=1}^N \langle X_i,v\rangle^2 - \E\langle X,v\rangle^2\right| &\le \E\sup_{v\in S^{d-1}}\left|\frac{1}{N\lambda}\sum_{i=1}^N \psi(\lambda\langle X_i,v\rangle^2) - \E\langle X,v\rangle^2 \right| 
\\
&\qquad+ \E\sup_{v\in S^{d-1}}\frac{1}{N}\sum_{i=1}^N \langle X_i,v\rangle^2 \ind{\lambda\langle X_i,v\rangle^2>1}.
\end{align*}
We choose $\lambda = \frac{1}{\kappa(4)^2\|\Sigma\|} \sqrt{\frac{\mathbf{r}(\Sigma)}{N}}$. By Proposition \ref{prop:trimmedprocess} we have
\[
\E\sup_{v\in S^{d-1}}\left|\frac{1}{N\lambda}\sum_{i=1}^N \psi(\lambda\langle X_i,v\rangle^2) - \E\langle X,v\rangle^2 \right| \lesssim \kappa(4)^2\|\Sigma\|\sqrt{\frac{\mathbf{r}(\Sigma)}{N}}.
\]
We proceed with the remaining term. Define the random set $I_v = \{i\in [N]: \langle X_i,v\rangle^2> \lambda^{-1}\}$. Let $m =\sup\limits_{v \in S^{d - 1}}|I_v|$. By \eqref{eq:relationfanda} we have
\begin{equation}
\label{eq:twoineqs}
\frac{m}{N\lambda} \le \sup_{v\in S^{d-1}}\frac{1}{N}\sum\nolimits_{i=1}^N \langle X_i,v\rangle^2 \ind{\lambda\langle X_i,v\rangle^2>1} \le \frac{1}{N}f(m, [N]).
\end{equation}
Observe that for any $p > 4$, there is $C(p) > 0$ such that for all $x \ge 1$,
\[
x^{1 - 4/(4+p)}\log^4 x \le C(p)\sqrt{x}.
\]
We now want to apply Theorem \ref{thm:rearrengmentthm} with the (random) value $m$. First, we assume that $\mathbf{r}(\Sigma) \le m \le N/c$.
In this case, by Theorem \ref{thm:rearrengmentthm}, with probability at least $1-\frac{c(p)}{N}$, it holds that
\begin{equation*}
\frac{1}{N}f(m, [N]) \le C(p)\left(\frac{1}{N}\max_{i\le N}\|X_i\|^2 + \|\Sigma\| \left(\frac{m}{N}\right)^{1/2}\right).
\end{equation*}
By \eqref{eq:twoineqs} and since $\|X_i\| \le (N\tr(\Sigma)\|\Sigma\|)^{1/4}$ we have on the same event 
\begin{equation*}
m \le C(p) \lambda \left(\max_{i\le N}\|X_i\|^2 + \sqrt{mN} \|\Sigma\|\right) \le C(p)\frac{1}{\|\Sigma\|}\sqrt{\frac{\mathbf{r}(\Sigma)}{N}}\left(\sqrt{mN} \|\Sigma\| + \sqrt{\tr(\Sigma)\|\Sigma \|N}\right) .
\end{equation*}
Solving the inequality above with respect to $m$ we have on the same event
\begin{equation*}
m \le C(p)\mathbf{r}(\Sigma).
\end{equation*}
We consider the case where $m$ does not satisfy $\mathbf{r}(\Sigma) \le m \le N/c$. If $m < \mathbf{r}(\Sigma)$, then we recover the same upper bound as above. If $m > N/c$, then on the event of Theorem \ref{thm:rearrengmentthm} for $k = \left\lfloor N/c -1\right\rfloor$ by \eqref{eq:twoineqs} and monotonicity we have
\[
\frac{k}{\lambda} \le f(k, [N]) \le C(p)\left(\sqrt{N\tr(\Sigma)\|\Sigma\|} + \|\Sigma\| \left(kN\right)^{1/2}\right).
\]
For our choice of $\lambda$ this bound leads to contradiction if $N \ge c(p)\mathbf{r}(\Sigma)$ for large enough $c(p)$. 
We are now ready to plug our bound $m \le C(p)\mathbf{r}(\Sigma)$ in \eqref{eq:twoineqs} to obtain that, with probability at least $1-\frac{c(p)}{N}$,
\[
\frac{1}{N}f(C(p)\mathbf{r}(\Sigma), [N]) \le C(p)\|\Sigma\|\sqrt{\frac{\mathbf{r}(\Sigma)}{N}}.
\]
Finally, for any $v \in S^{d - 1}$, we have $|\langle X_i,v\rangle| \le \|X_i\| \le (N\tr(\Sigma)\|\Sigma\|)^{1/4}$. Therefore,
\begin{align*}
\E \sup_{v\in S^{d-1}}\frac{1}{N}\sum_{i=1}^N \langle X_i,v\rangle^2 \ind{\lambda\langle X_i,v\rangle^2>1} &\le C(p)\|\Sigma\|\sqrt{\frac{\mathbf{r}(\Sigma)}{N}} + \frac{c(p)}{N}\sqrt{N\tr(\Sigma)\|\Sigma\|}
\\
&=(C(p) + c(p))\|\Sigma\|\sqrt{\frac{\mathbf{r}(\Sigma)}{N}}.
\end{align*}
The claim of Theorem \ref{thm:baiyin} follows.
\qed

\section{Proof of Theorem \ref{thm:informalmain} in the regime \texorpdfstring{$p = 4$}{Lg}}
\label{sec:pfour}
First, we provide a high probability version of Proposition \ref{prop:trimmedprocess}. In the definition of our estimator we assume that both $\|\Sigma\|$ and $\tr(\Sigma)$ are known, so that our estimator depends on such quantities. In Section \ref{sec:estimtraceandopernorm}, we provide optimal guarantees for estimating these parameters. The following result can be seen as a special case of Lemma 5 in \cite{zhivotovskiy2021dimension}. We provide a detailed sketch of the proof for the sake of completeness.
\begin{proposition}
\label{prop:uniftruncation}
Assume that $X$ is a zero mean random vector with covariance $\Sigma$ satisfying the $L_4-L_2$ \emph{norm equivalence}. Let $X_1, \ldots, X_N$ be a sample of independent copies of $X$. For a fixed truncation level $\lambda > 0$, with probability at least $1 - \delta$, it holds that
\[
\sup_{v\in S^{d-1}}\left|\frac{1}{N\lambda}\sum_{i=1}^N \psi(\lambda\langle X_i,v\rangle^2) - \E\langle X,v\rangle^2 \right|
\lesssim \frac{\mathbf{r}(\Sigma) + \log(1/\delta)}{\lambda N} + \lambda\kappa^{4}\|\Sigma\|^2,
\]
where $\kappa = \kappa(4)$.
\end{proposition}
\begin{proof}
We repeat the lines of the proof of Proposition \ref{prop:trimmedprocess}. However, instead we plug 
\[
f(X, \theta, \nu) = \log\left(1 + \lambda\langle X, \theta\rangle\langle X, \nu\rangle  + c\lambda^2\left(\langle X, \theta\rangle\langle X, \nu\rangle\right)^2 \right)
\]
in the first part of Lemma \ref{lem:pacbayes}. This implies that, with probability at least $1 - \delta$,
\begin{equation}
\label{eq:uniformboundwithtrimming}
\sup\limits_{v \in S^{d - 1}}\left(\frac{1}{N}\sum\limits_{i = 1}^N \E_{\rho_v}f\left(X_i, \theta, \nu\right) - \lambda\E\langle X, v\rangle^2\right) \le \frac{\mathbf{r}(\Sigma) + \log(1/\delta)}{ N} + 4c\lambda^2\kappa^{4}\|\Sigma\|^2,
\end{equation}
where $c > 0$ is the same constant as in the proof of Proposition \ref{prop:trimmedprocess}. Furthermore, by the Bernstein inequality we have
\begin{align*}
\frac{1}{N}\sum\limits_{i = 1}^N\min\{1, \lambda^2(\mathbf{r}(\Sigma))^{-2}\|X_i\|^4/3\} &\le \E\min\{1, 2\lambda^2(\mathbf{r}(\Sigma))^{-2}\|X\|^4/6\} 
\\
&\ + \sqrt{\frac{2\log(1/\delta)}{N}\E\min\{1, 2\lambda^2(\mathbf{r}(\Sigma))^{-2}\|X\|^4/6\}} + \frac{2\log(1/\delta)}{3N}.
\end{align*}
where we used that each summand belongs to the interval $[0, 1]$ and therefore the variance of each summand is bounded by its expectation. Following the proof of Lemma \ref{lem:normofthevec}, we have $\E\|X\|^4 \le \kappa^4(\tr(\Sigma))^2$. This implies that, with probability at least $1 - \delta$,
\begin{equation}
\label{eq:bernstein}
\frac{1}{N}\sum\limits_{i = 1}^N\min\{1, 2\lambda^2(\mathbf{r}(\Sigma))^{-2}\|X_i\|^4/6\} \le \frac{2\lambda^2\kappa^4\|\Sigma\|^2}{3} + \frac{3\log(1/\delta)}{N}.
\end{equation}
Using this line and the union bound together with \eqref{eq:uniformboundwithtrimming} and the derivations in the proof of Proposition \ref{prop:trimmedprocess}, we obtain the one-sided version of our claim. Repeating the same lines by replacing $\lambda$ by $-\lambda$ and using the union bound again, we finish the proof.
\end{proof}

Our next result concludes the proof of Theorem 1 when $p = 4$. Recall that $\kappa$ denotes $\kappa(4)$.

\begin{theorem}
\label{thm:thecasepfour}
There is an absolute constant $C > 0$ such that the following holds.
Assume that $X$ is a zero mean random vector with covariance $\Sigma$ satisfying the $L_4-L_2$ norm equivalence. Fix the corruption level $\eta \in [0, 1]$ and the confidence level $\delta \in (0, 1)$. Assume that $\widetilde{X}_{1}, \ldots, \widetilde{X}_N$ is an $\eta$-corrupted sample. Then there exists an estimator $\widehat{\Sigma}_{\eta, \delta}$ such that, with probability at least $1 - \delta$,
\[
\left\|\widehat{\Sigma}_{\eta, \delta} - \Sigma\right\| \le C\kappa^2\|\Sigma\|\left(\sqrt{\frac{\mathbf{r}(\Sigma) + \log(1/\delta)}{N}} + \sqrt{\eta}\right).
\]
\end{theorem}

We begin with presenting the estimator. As we mentioned, we assume that $\eta, \|\Sigma\|, \tr(\Sigma)$ and the constant $\kappa = \kappa(4)$ are known. In Section \ref{sec:estimtraceandopernorm}, we discuss how to estimate $\|\Sigma\|$ and $\tr(\Sigma)$ up to multiplicative constant factors. The dependence on $\kappa$ in our estimator can also be ignored, though in this case we obtain a slightly weaker dependence on this parameter in the final bound. 

\medskip
\begin{tcolorbox}
\label{Box:firstEstimator}
The estimator in Theorem \ref{thm:thecasepfour}
\begin{enumerate}
    \item Given $\delta, \tr(\Sigma), \|\Sigma\|, \kappa$, and an $\eta$-corrupted sample $\widetilde{X}_1,\ldots,\widetilde{X}_{N}$ we set
    \[
    \lambda = \frac{1}{\kappa^2\|\Sigma\|}\sqrt{\frac{\mathbf{r}(\Sigma) + \log(1/\delta) + \eta N}{N}}.
    \]
    \item Define \[
\Gamma = \bigcap\limits_{v \in S^{d - 1}}\left\{A \in \mathbb{S}_{+}^d: \left|\frac{1}{\lambda N}\sum\limits_{i=1}^N \psi\left(\lambda\langle\widetilde{X}_i,v\rangle^2\right) - v^{\top}A v\right| \le C\lambda\kappa^4\|\Sigma\|^2/2\right\}.
\]
    \item 
    Let $\widehat{\Sigma}_{\eta,\delta}$ be any matrix in the set $\Gamma$. If the set is empty, we output $\widehat{\Sigma}_{\eta,\delta} = 0$.
\end{enumerate}
\end{tcolorbox}

\begin{proof}
Since the truncation function $\psi(\cdot)$ is bounded, we have
\[
\sup_{v\in S^{d-1}}\left|\frac{1}{N\lambda}\sum_{i=1}^N \psi(\lambda\langle \widetilde{X}_i,v\rangle^2) - \E\langle X,v\rangle^2 \right| \le \sup_{v\in S^{d-1}}\left|\frac{1}{N\lambda}\sum_{i=1}^N \psi(\lambda\langle X_i,v\rangle^2) - \E\langle X,v\rangle^2 \right| + \frac{\eta}{\lambda}.
\]
Combining this with Proposition \ref{prop:uniftruncation} we have for some $C > 0$, with probability at least $1 - \delta$,
\[
\sup_{v\in S^{d-1}}\left|\frac{1}{N\lambda}\sum_{i=1}^N \psi(\lambda\langle \widetilde X_i,v\rangle^2) - v^{\top}\Sigma v \right| \le \frac{C}{4}\left(\frac{\mathbf{r}(\Sigma) + \log(1/\delta) + \eta N}{\lambda N} + \lambda\kappa^{4}\|\Sigma\|^2\right).
\]
Using the triangle inequality, the definition of $\widehat{\Sigma}_{\eta, \delta}$ and the line above, we have on the same event
\begin{align*}
\left\|\widehat{\Sigma}_{\eta, \delta} - \Sigma\right\| &\le \sup\limits_{v \in S^{d-1}}\left|\frac{1}{N\lambda}\sum_{i=1}^N \psi(\lambda\langle \widetilde X_i,v\rangle^2) - v^{\top}\widehat{\Sigma}_{\eta, \delta}v\right| + \sup\limits_{v \in S^{d-1}}\left|\frac{1}{N\lambda}\sum_{i=1}^N \psi(\lambda\langle \widetilde X_i,v\rangle^2) - v^{\top}\Sigma v\right|
\\
&\le C\kappa^2\|\Sigma\|\left(\sqrt{\frac{\mathbf{r}(\Sigma) + \log(1/\delta)}{N}} + \sqrt{\eta}\right).
\end{align*}
Our choice of $\lambda$ concludes the proof.
\end{proof}

\begin{remark}
It is straightforward to verify that the matrix $\widehat{\Sigma}_{\eta, \delta}$ defined by \eqref{eq:minmaxestimator} satisfies the assumption of Theorem \ref{thm:thecasepfour} and can be used as an estimator of the covariance matrix. 
\end{remark}
Finally, we show how to estimate the value of the truncation parameter $\lambda$ using the corrupted observations.

\subsection{Estimating the truncation level \texorpdfstring{$\lambda$}{Lg} in the regime \texorpdfstring{$p = 4$}{Lg}}
\label{sec:estimtraceandopernorm}

We begin this section with a brief discussion of our contamination model. To simplify the technical aspects of our analysis, we impose a slight restriction on the strong contamination model described in the introduction. In our analysis, we require the estimation of multiple parameters, which can be accomplished by dividing the sample into several equal parts and estimating these parameters using independent sub-samples. This approach is standard, and has been analyzed, for example, in \cite{mendelson2020robust} in the same context. However, in the strong contamination model, the adversary has the ability to replace some observations within each sub-sample in a manner that depends on the entire sample. This can result in dependent sub-samples and a more complicated stability type analysis appearing in the literature (see \cite[Section 6]{diakonikolas2022outlier}). We therefore restrict the adversary from making the sub-samples dependent on each other. We note that this assumption is standard and has been made (implicitly) in some recent works \cite{lugosi2021robust,oliveira2022improved}. Furthermore, it is automatically satisfied in standard contamination models, including the Huber model.

An important aspect of the proof of Theorem \ref{thm:thecasepfour} is that we only need to know $\|\Sigma\|$ and $\tr(\Sigma)$ up to multiplicative constant factors. That is, we want to find two estimators $\widehat \varphi_1$ and $\widehat \varphi_2$ such that for some $c > 1$, with high probability with respect to the realization of the corrupted sample $\widetilde X_1, \ldots, \widetilde X_N$, 
\[
\frac{1}{c}\tr(\Sigma) \le \widehat \varphi_1 \le c\tr(\Sigma) \quad \text{and} \quad \frac{1}{c}\|\Sigma\| \le \widehat \varphi_2 \le c\|\Sigma\|.
\]
If these estimators are available, we can assume, without the loss of generality, that our initial sample is of size $3N$. We use the first $2N$ elements to estimate $\tr(\Sigma)$ and $\|\Sigma\|$ and then plug them into the truncation level $\lambda$ in Theorem \ref{thm:thecasepfour}. A similar sample splitting strategy is used in \cite{mendelson2020robust}. Our contribution is that we provide estimators $\widehat \varphi_1$ and $\widehat \varphi_2$ such that they take the adversarial corruption into account as well as do not contain an unnecessary logarithmic factor. For the rest of the section, we fix $\kappa = \kappa(4)$, where $\kappa(\cdot)$ is defined in \eqref{eq:momeqv}. 

\subsection{Estimating \texorpdfstring{$\tr(\Sigma)$}{Lg}}

One can estimate $\tr(\Sigma)$ using, for example, the trimmed mean estimator analyzed in \cite{lugosi2021robust}. Let $e_1, \ldots, e_d$ denote the canonical basis in $\mathbb{R}^d$. Observe that 
\[
\tr(\Sigma) = \sum\nolimits_{i = 1}^d\E\langle X, e_i\rangle^2. 
\]
Using the $L_4-L_2$ norm equivalence and the proof of Lemma \ref{lem:normofthevec} we have $\var\left(\sum\nolimits_{i = 1}^d\langle X, e_i\rangle^2\right) \le \kappa^4(\tr(\Sigma))^2$. Given an $\eta$-corrupted sample of size $2N$, the trimmed mean estimator $\widehat \varphi_1$ applied to the random variable $\sum\nolimits_{i = 1}^d\langle X, e_i\rangle^2$ implies that, for any $\delta \in (\exp(-N)/4, 1)$, with probability at least $1 - 4\exp(-\varepsilon N/12)$,
\[
|\widehat \varphi_1 - \tr(\Sigma)| \le 10\sqrt{\varepsilon}\kappa^2\tr(\Sigma) ,
\]
where $\varepsilon = 8\eta + 12\frac{\log(4/\delta)}{N}$.
This bound is presented in \cite[Theorem 1]{lugosi2021robust}.
For the sake of brevity, we do not provide the details of the multivariate trimmed mean estimator. Observe that for this choice of $\varepsilon$, it holds that $1 - 4\exp(-\varepsilon N/12) \ge 1 - \delta$. Provided that $10\sqrt{\varepsilon}\kappa^2 \le \frac{1}{2}$ we obtain that, with probability at least $1 - \delta$,
\begin{equation}
\label{eq:estuptoconst}
\frac{1}{2}  \tr(\Sigma) \le \widehat \varphi_1 \le 2 \tr(\Sigma).
\end{equation}
It is only left to check that it is sufficient to have $\eta \in \left[0, \frac{1}{40^2\kappa^2}\right]$ and $N \ge 12\cdot 40^2\kappa^4\log(4/\delta)$. \subsection{Estimating \texorpdfstring{$\|\Sigma\|$}{Lg}}
\label{sec:opernorm}
A simply way of estimating $\|\Sigma\|$ is to first estimate the full covariance matrix $\Sigma$ as in Theorem \ref{thm:thecasepfour} and then compute the operator norm of the estimator. 
However, the difficulty of estimating $\|\Sigma\|$ is that we are not allowed to choose the truncation level $\lambda$ depending on $\|\Sigma\|$. In particular, all previous results we are aware of lead to an additional logarithmic factor in the assumption on the sample size (see \cite{mendelson2020robust, catoni2017dimension}). Instead we provide an adaptive estimator similar in spirit to the original robust estimator of O. Catoni \cite{catoni2012challenging} for estimating the mean of a random variable (see also \cite{catoni2016pac} for related computations). To simplify the proof, we assume that the distribution of $X$ satisfies $\Pr(X = 0) = 0$. This assumption is mild and can always be satisfied if we add a small Gaussian perturbation to $X$ without changing the covariance matrix of $X$ too much.

\begin{proposition}
\label{prop:estlargesteigenvalue}
Assume that $X$ is a zero mean random vector with covariance $\Sigma$ satisfying the $L_4-L_2$ norm equivalence. Assume additionally that $\Pr(X = 0) = 0$. Let $c \ge 1$ be a large enough absolute constant. Fix the corruption level $\eta \in [0, \frac{1}{300c\kappa^4}]$ and the confidence level $\delta \in (0, 1/4)$. Assume that $\widetilde{X}_{1}, \ldots, \widetilde{X}_N$ is an $\eta$-corrupted sample. Then there is a unique value $\widehat{\alpha} > 0$ satisfying
\[
\frac{1}{N}\sup\limits_{v \in S^{d - 1}}\sum\nolimits_{i = 1}^N \psi\left(\widehat\alpha^2\langle \widetilde X_i,  v\rangle^2\right) = \frac{1}{20c\kappa^4} + \eta.
\]
If $N \ge 100c\kappa^4\mathbf{r}(\Sigma)+  400c\kappa^4\log(1/\delta)$, then with probability at least $1 - 4\delta$, it holds that
\[
\frac{1}{4} \|\Sigma\| \le \frac{1}{24c\kappa^4\widehat{\alpha}^2} \le 4\|\Sigma\|.
\]
\end{proposition}

\begin{proof}
We begin with the analysis of an uncorrupted sample $X_1, \ldots, X_N$.
Our first aim is to choose the distributions $\mu$ and $\rho$ in Lemma \ref{lem:pacbayes}. Let 
$
\Theta = (\mathbb{R}^d)^2
$
and choose $\mu$ to be a product of two zero mean multivariate Gaussians with mean zero and covariance $\beta^{-1}I_d$. For $v \in S^{d - 1}$ and $\alpha > 0$ let $\rho_{\alpha, v}$ be a product of two multivariate Gaussian distribution with mean $\alpha v$ and covariance $\beta^{-1}I_d$. Because of this, if $(\theta, \nu)$ is distributed according to $\rho_{\alpha, v}$, we have $\E_{\rho_{v}}(\theta, \nu) = (\alpha v, \alpha v)$. By the additivity of $\mathcal{KL}$-divergence for product measures and the standard formula, we have
\[
\mathcal{KL}(\rho_{\alpha, v}, \mu) = \alpha^2\beta.
\]
For the rest of the proof we sometimes write $\rho$ instead of $\rho_{\alpha, v}$.  By the first part of Lemma \ref{lem:almostconvex}, we have 
\begin{align*}
\psi\left(\alpha^2\langle X,  v\rangle^2\right) &=\psi\left(\E_{\rho}\langle X, \theta\rangle\langle X, \nu\rangle\right)
\\
&\le \E_{\rho}\log\left(1 + \langle X, \theta\rangle\langle X, \nu\rangle  + \left(\langle X, \theta\rangle\langle X, \nu\rangle\right)^2 \right) + \min\{1, \E_{\rho}\left(\langle X, \theta\rangle\langle X, \nu\rangle\right)^2/6\}.
\end{align*}
Observe that 
\[
\E_{\rho}\left(\langle X, \theta\rangle\langle X, \nu\rangle\right)^2 = (\alpha^2\langle X, v\rangle^2 + \beta^{-1}\|X\|^2)^2 \le 2\alpha^4\langle X, v\rangle^4 + 2\beta^{-2}\|X\|^4,
\]
and we can write
\[
\min\{1, \E_{\rho}\left(\langle X, \theta\rangle\langle X, \nu\rangle\right)^2/6\} \le \min\{1, 2\alpha^4\langle X, v\rangle^4/6\} + \min\{1, 2\beta^{-2}\|X\|^4/6\}.
\]
Repeating the lines of the proof of Proposition \ref{prop:trimmedprocess}, we get
\[
\E_{\rho}\log\E\left(1 + \langle X, \theta\rangle\langle X, \nu\rangle  + c\left(\langle X, \theta\rangle\langle X, \nu\rangle\right)^2 \right) \le \alpha^2 v^{\top}\Sigma v + 2c\kappa^{4}(\alpha^4\|\Sigma\|^2 + \beta^{-2}(\tr(\Sigma))^2),
\]
where $c \ge 1$ is an absolute constant. Using the first part of Lemma \ref{lem:pacbayes} we obtain that, with probability at least $1 - \delta$, simultaneously for all $v \in S^{d - 1}$ and $\alpha > 0$,
\begin{align*}
\frac{1}{N}\sum\nolimits_{i = 1}^N \psi\left(\alpha^2\langle X_i,  v\rangle^2\right) &\le \alpha^2 v^{\top}\Sigma v + 2c\kappa^{4}(\alpha^4\|\Sigma\|^2 + \beta^{-2}(\tr(\Sigma))^2) 
\\
&\qquad+ \frac{\alpha^2 \beta + \log(1/\delta)}{N} + \frac{1}{N}\sum\nolimits_{i = 1}^N\min\{1, 2\beta^{-2}\|X_i\|^4/6\}.
\end{align*}
Using the Bernstein inequality as in \eqref{eq:bernstein} we have, with probability at least $1 - \delta$,
\[
\frac{1}{N}\sum\nolimits_{i = 1}^N\min\{1, 2\beta^{-2}\|X_i\|^4/6\} \le \frac{2\kappa^4\beta^{-2}(\tr(\Sigma))^2}{3} + \frac{3\log(1/\delta)}{N}.
\]
We choose $\beta = 10c\kappa^4 \tr(\Sigma)$ and using the union bound, we obtain that, with probability at least $1 - 2\delta$, simultaneously for all $v \in S^{d - 1}$ and $\alpha > 0$,
\[
\frac{1}{N}\sum\nolimits_{i = 1}^N \psi\left(\alpha^2\langle X_i,  v\rangle^2\right) \le \alpha^2 v^{\top}\Sigma v + 2c\kappa^{4}\alpha^4\|\Sigma\|^2 + \frac{2}{75c\kappa^4} + \frac{10c\kappa^4\alpha^2\tr(\Sigma)}{N} + \frac{4\log(1/\delta)}{N}.
\]
Since $N \ge 100c\kappa^4\mathbf{r}(\Sigma)+  400c\kappa^4\log(1/\delta)$, we have on the same event
\[
\frac{1}{N}\sum\nolimits_{i = 1}^N \psi\left(\alpha^2\langle X_i,  v\rangle^2\right) \le \alpha^2 v^{\top}\Sigma v + \frac{1}{10}\alpha^2\|\Sigma\|+ 2c\kappa^{4}\alpha^4\|\Sigma\|^2 + \frac{11}{300c\kappa^4}.
\]
Returning to the $\eta$-corrupted sample $\widetilde X_1, \ldots, \widetilde X_N$ and since $|\psi(\cdot)| \le 1$,
\begin{equation}
\label{eq:sumofcorrupt}
\frac{1}{N}\sum\nolimits_{i = 1}^N \psi\left(\alpha^2\langle \widetilde X_i,  v\rangle^2\right) \le \alpha^2 v^{\top}\Sigma v + \frac{1}{10}\alpha^2\|\Sigma\|+ 2c\kappa^{4}\alpha^4\|\Sigma\|^2 + \frac{11}{300c\kappa^4} + \eta.
\end{equation}
Observe that the function $\varphi: \mathbb{R}_{+} \to \mathbb{R}_{+}$ given by
\[
\varphi(\alpha) =  \frac{1}{N}\sup\limits_{v \in S^{d - 1}}\sum\nolimits_{i = 1}^N \psi\left(\alpha^2\langle \widetilde X_i,  v\rangle^2\right) 
\]
is continuous (it is easy to check that $\varphi(\cdot)$ is a supremum over equi-Lipschitz functions) and non-decreasing. Furthermore, there is $w \in S^{d - 1}$ such that $\langle X_i,  w\rangle \neq 0$ for all $i \in [N]$. Indeed, to construct such a vector we pick a random vector $U$ distributed as a zero mean multivariate Gaussian random vector with unit covariance. Since $X_i \neq 0$ almost surely for all $i \in [N]$ we have that $\langle U, X_i\rangle$ is a zero mean Gaussian (conditionally on $X_i$) with non-zero covariance. Thus, taking $w = U/\|U\|$ we have almost surely $\langle X_i,  w\rangle \neq 0$ for all $i \in [N]$.
We have $\varphi(0) = 0$ and $\varphi\left(\frac{1}{\min_{i \in [N]}|\langle X_i,  w\rangle|}\right) \ge 1 - \eta$ since at most $\eta N$ vectors $\widetilde{X}_i$ can be orthogonal to $w$.
Therefore, since $\kappa, c \ge 1$ and $\eta \le \frac{1}{300c\kappa^4}$ there is $\widehat{\alpha} > 0$ such that
\[
\varphi(\widehat{\alpha}) = \frac{1}{20c\kappa^4} + \eta. 
\]  
Using \eqref{eq:sumofcorrupt}, we get on the same event
\[
2c\kappa^{4}\widehat{\alpha}^4\|\Sigma\|^2 + \frac{11}{10}\widehat{\alpha}^2\|\Sigma\| - \frac{1}{75c\kappa^4} \ge 0.
\]
Solving this as a quadratic equation (only the positive root plays a role) with respect to $\widehat{\alpha}^2\|\Sigma\|$ we get
\[
\frac{-11/10 + \sqrt{(11/10)^2 + 8/75}}{4c\kappa^{4}\widehat{\alpha}^2} \le \|\Sigma\|.
\]
This provides a lower bound on $\|\Sigma\|$ in terms of $\widehat{\alpha}^2$.

We proceed with an upper bound. For $v \in S^{d - 1}$ and $\alpha > 0$ let $\rho_{\alpha, -\alpha, v}$ be a product of two multivariate Gaussian distributions with means $\alpha v$ and $-\alpha v$ respectively and the same covariance matrix $\beta^{-1}I_d$. Repeating the proof, we obtain the following analog of inequality \eqref{eq:sumofcorrupt}. With probability at least $1 - 2\delta$, for all $v \in S^{d - 1}$ and $\alpha > 0$,
\begin{equation}
\label{eq:lowertailquadraticeq}
-\frac{1}{N}\sum\nolimits_{i = 1}^N \psi\left(\alpha^2\langle \widetilde X_i,  v\rangle^2\right) \le -\alpha^2 v^{\top}\Sigma v + \frac{1}{10}\alpha^2\|\Sigma\|+ 2c\kappa^{4}\alpha^4\|\Sigma\|^2 + \frac{11}{300c\kappa^4} + \eta.
\end{equation}
Let $v_1$ be a maximizer of $v^{\top}\Sigma v$ in $S^{d - 1}$. We have $v_1^{\top}\Sigma v_1 =\|\Sigma\|$. Furthermore, we have
\[
-\sup\limits_{v \in S^{d - 1}}\frac{1}{N}\sum\nolimits_{i = 1}^N \psi\left(\alpha^2\langle \widetilde X_i,  v\rangle^2\right) \le -\frac{1}{N}\sum\nolimits_{i = 1}^N \psi\left(\alpha^2\langle \widetilde X_i,  v_1\rangle^2\right).
\]
Since $-\varphi(\cdot)$ is non-increasing, this implies that simultaneously for all $\alpha \in [0, \widehat{\alpha}]$ on the event where \eqref{eq:lowertailquadraticeq} holds, we have
\[
-\frac{1}{20c\kappa^4} - \eta \le -\frac{9}{10}\alpha^2 \|\Sigma\|+ 2c\kappa^{4}\alpha^4\|\Sigma\|^2 + \frac{11}{300c\kappa^4} + \eta.
\]
Since $\eta \le \frac{1}{300c\kappa^4}$ we reduce this equation to 
\begin{equation}
\label{eq:onalpha}
2c\kappa^{4}\alpha^4\|\Sigma\|^2 -\frac{9}{10}\alpha^2 \|\Sigma\| + \frac{7}{75c\kappa^4} \ge 0.
\end{equation}
As a function of $\alpha^2 \|\Sigma\|$ this quadratic equation has two non-negative roots:
\[
x_1 = \frac{9/10 - \sqrt{(9/10)^2 - 56/75}}{4c\kappa^4}\quad \text{and} \quad x_2 = \frac{9/10 + \sqrt{(9/10)^2 - 56/75}}{4c\kappa^4}.
\]
We show that since \eqref{eq:onalpha} holds for all $\alpha \in [0, \widehat{\alpha}]$, we should have that $\widehat{\alpha}^2\|\Sigma\| \le x_1$. Indeed, assume that instead $\widehat{\alpha}^2\|\Sigma\| \ge x_2$. In this case $0 \le (x_{1} + x_2)/2 \le \widehat{\alpha}^2\|\Sigma\|$ and thus \eqref{eq:onalpha} should be satisfied for $\alpha^2\|\Sigma\| = (x_{1} + x_2)/2$. The obtained contradiction proves that $\widehat{\alpha}^2\|\Sigma\| \le x_1$. Finally, by union bound we have, with probability at least $1 - 4\delta$, 
\[
\frac{-11/10 + \sqrt{(11/10)^2 + 8/75}}{4c\kappa^{4}\widehat{\alpha}^2} \le \|\Sigma\| \le \frac{9/10 - \sqrt{(9/10)^2 - 56/75}}{4c\kappa^4\widehat{\alpha}^2}.
\]
Algebraic computations conclude the proof.
\end{proof}
We end up by observing that the estimator proposed in this section admits a straightforward generalization for (uncontaminated) higher order tensors. In this setup one is aiming to estimate $\E\langle X, v\rangle^s$ uniformly over $v \in S^{d -1}$ for some integer $s \ge 2$, as in the recent work of Mendelson \cite{mendelson2021approximating}. In order to provide this extension we apply \cite[Lemma 5]{zhivotovskiy2021dimension} together with the estimators of the trace and the operator norm provided above. We note that getting the optimal dependence on the contamination level $\eta$ seems more subtle when $s > 2$.

\section{Proof of Theorem \ref{thm:informalmain} in the regime \texorpdfstring{$p > 4$}{Lg}}
\label{sec:pmorefour}
In this section, we prove Theorem \ref{thm:informalmain} in the regime $p>4$. We begin with the following symmetrized version of Theorem \ref{thm:baiyin}.
\begin{corollary}
\label{cor:baiyin}
Assume that $Y$ is a zero mean random vector with covariance $\Sigma$ satisfying the $L_p-L_2$ \emph{norm equivalence} with $p > 4$. Let $Y_1, \ldots, Y_N$ be a sample of independent copies of $Y$. Consider the truncated vectors $X_i = Y_i\ind{\|Y_i\| \le (N\tr(\Sigma)\|\Sigma\|)^{1/4}}$ for $i = 1, \ldots, N$. Let $\varepsilon_1, \ldots, \varepsilon_N$ be independent Rademacher random signs. If $N  \ge c(p)\mathbf{r}(\Sigma)$, then it holds that
\[
\E\sup\limits_{v \in S^{d - 1}}\left|\frac{1}{N}\sum\limits_{i = 1}^N \varepsilon_i\left\langle X_i, v\right\rangle^2\right| \le C(p)\|\Sigma\|\sqrt{\frac{\mathbf{r}(\Sigma)}{N}},
\]
where $c(p)$ and $C(p)$ depend only on $p$ and $\kappa(p)$ and the expectation is taken with respect to both $X_i$ and $\varepsilon_i$, $i = 1, \ldots, N$.
\end{corollary}
\begin{remark}
We highlight that the result of Theorem \ref{thm:baiyin} and Corollary \ref{cor:baiyin} admits extensions to higher order power in the spirit of \cite{vershynin2011approximating}. For the related derivations in the light-tailed case, see the proof of Theorem 2 in \cite{zhivotovskiy2021dimension}.
\end{remark}

\begin{proof}
We use the standard desymmetrization argument \cite[Section 2.1]{koltchinskii2011oracle}. Let $X^{\prime}_1, \ldots, X^{\prime}_N$ be the independent copies of $X$ and let $\E^{\prime}$ denote the expectation with respect to these random vectors. Using the triangle inequality, Jensen's inequality and the contraction principle \cite[Theorem 4.4]{ledoux2013probability} we get
\begin{align*}
\E\sup\limits_{v \in S^{d - 1}}\left|\sum\limits_{i = 1}^N \varepsilon_i\left\langle X_i, v\right\rangle^2\right| &\le \E\sup\limits_{v \in S^{d - 1}}\left|\sum\limits_{i = 1}^N \varepsilon_i(\left\langle X_i, v\right\rangle^2 - \E^{\prime}\left\langle X^{\prime}_i, v\right\rangle^2)\right| + \E\sup\limits_{v \in S^{d - 1}}\left|\sum\limits_{i = 1}^N \varepsilon_i\E\left\langle X_i, v\right\rangle^2\right|
\\
&\le\E\E^{\prime}\sup\limits_{v \in S^{d - 1}}\left|\sum\limits_{i = 1}^N \left\langle X_i, v\right\rangle^2 - \left\langle X^{\prime}_i, v\right\rangle^2\right| + \|\Sigma\|\sqrt{N}
\\
&\le2\E\sup\limits_{v \in S^{d - 1}}\left|\sum\limits_{i = 1}^N \left\langle X_i, v\right\rangle^2 - \E\left\langle X_i, v\right\rangle^2\right| + \|\Sigma\|\sqrt{N}.
\end{align*}
The bound of Theorem \ref{thm:baiyin} concludes the proof.
\end{proof}

In the regime $p > 4$, we show how to adapt the proposed estimator to obtain the optimal dependence on the contamination level $\eta$. The main insight is to interpret our estimator as a second order version of the trimmed mean estimator of Lugosi and Mendelson \cite{lugosi2021robust}. Our proof mainly follows their steps, however two important modifications are required: First, we need to exploit Theorem \ref{thm:baiyin} when controlling the expected supremum of quadratic processes, while in \cite{lugosi2021robust} the control of the corresponding process follows from a simple application of the Cauchy-Schwartz inequality. To do so, we also need to truncate the norms of our observations as explained after Proposition \ref{prop:truncdoesnothurt}.  Second, since the quadratic forms of interest are non-negative, we consider only a one-sided truncation. The main theorem of this section is the following result.

\begin{theorem}
\label{thm:thecasep_greater_four}
Assume that $X$ is a zero mean random vector with covariance $\Sigma$ satisfying the $L_p-L_2$ norm equivalence with $p>4$. Fix the corruption level $\eta \in [0, 1]$ and the confidence level $\delta \in (0, 1)$. Then there exists an estimator $\widehat{\Sigma}_{\eta, \delta}$ such that, with probability at least $1 - \delta$,
\[
\left\|\widehat{\Sigma}_{\eta, \delta} - \Sigma\right\| \le C(p)\|\Sigma\|\left(\sqrt{\frac{\mathbf{r}(\Sigma) + \log(1/\delta)}{n}} + \kappa(p)^2\eta^{1-2/p}\right).\
\]
Here $C(p)$ is a non-increasing function of $p$ that satisfies $C(p) \to \infty$ as $p \to 4$.
\end{theorem}

Let us first recall the form of our estimator in the case where $p = 4$. For some specifically chosen $\lambda > 0$, the estimator in Theorem \ref{thm:thecasepfour} is defined by the following set
\[
\Gamma = \bigcap\limits_{v \in S^{d - 1}}\left\{A \in \mathbb{S}_{+}^d: \left|\frac{1}{\lambda N}\sum\limits_{i=1}^N \psi\left(\lambda\langle\widetilde{X}_i,v\rangle^2\right) - v^{\top}A v\right| \le C\lambda\kappa^4\|\Sigma\|^2/2\right\},
\]
where $\psi(\cdot)$ is the truncation function given by \eqref{eq:truncfunction}.
We defined $\widehat{\Sigma}_{\eta,\delta}$ to be any matrix in $\Gamma$.
In the regime $p > 4$, we introduce a  more involved estimator that uses a direction dependent value of the one-sided threshold. For simplicity, we assume that we observe a sample of size $2N$. For a set of positive semi-definite matrices $\Gamma$, we define its diameter as $\Delta(\Gamma)=\sup_{A,B \in \Gamma}\|A-B\|$.
\medskip
\begin{tcolorbox}
\label{Box:Estimator}
The estimator in Theorem \ref{thm:thecasep_greater_four}
\begin{enumerate}
    \item Given the confidence level $\delta$, the corruption level $\eta$ and the $\eta$-corrupted sample $\widetilde{Y}_1,\ldots,\widetilde{Y}_{2N}$, we split the sample in two parts of equal size: Truncate each vector $\widetilde{Z}_i=\widetilde{Y}_i \ind{\|\widetilde{Y}_i\|\le R}$ for $i\in \{1,\ldots,N\}$ and $\widetilde{X}_i=\widetilde{Y}_i \ind{\|\widetilde{Y}_i\|\le R}$ for $i\in \{N+1,\ldots,2N\}$, where $R=(N\tr(\Sigma)\|\Sigma\|)^{1/4}$.
    \item Set $\varepsilon = \max\left(20\eta, 560\frac{\log(2/\delta)}{N}\right)$.
    \item For the first half of the sample, define $q_v=(\langle \widetilde{Z}_i,v\rangle^2)_{N\varepsilon/2}^{\ast}$.
    \item For the second half of the sample $\widetilde{X}_1,\ldots,\widetilde{X}_N$, we proceed as follows: For every $Q>0, v\in S^{d - 1}$, define the trimming level $\lambda_v(Q) = (q_v +Q)^{-1}$ and set
    
      \begin{equation*}
        \Gamma(Q)= \bigcap\limits_{v\in S^{d-1}}\left\{A \in \mathbb{S}_{+}^d: \left|\frac{1}{\lambda_v(Q)N}\sum_{i=1}^N \psi\left(\lambda_v(Q)\langle\widetilde{X}_i,v\rangle^2\right) - v^{\top}A v\right| \le 4\varepsilon Q\right\}.
    \end{equation*}
    
    \item
    Let $\widehat{\Sigma}_{\eta,\delta}$ be any matrix in the set $\Gamma(2^{i^{\ast}})$, where $i^{\ast}$ minimizes the diameter $\Delta(\Gamma(2^{i}))$ over all integers $i$ subject to be non-empty.
\end{enumerate}
\end{tcolorbox}

\begin{remark}
Due to the bounds in Section \ref{sec:estimtraceandopernorm}, we explicitly assume the knowledge of the operator norm $\|\Sigma\|$ and the effective rank $\mathbf{r}(\Sigma)$ up to multiplicative constant factors.
\end{remark}
Our first technical observation is that, by truncating the norms of the vectors at the level $R = (N\tr(\Sigma)\|\Sigma\|)^{1/4}$, the covariance matrix does not change too much. At the same time, the truncation allows us to apply Theorem \ref{thm:baiyin} in our analysis.

\begin{proposition}
\label{prop:truncdoesnothurt}
Assume that $Y$ is a zero mean random vector with covariance $\Sigma$ satisfying the $L_p-L_2$ \emph{norm equivalence} with $p \ge 4$. Consider the truncated vector $X = Y\ind{\|Y\| \le (N\tr(\Sigma)\|\Sigma\|)^{1/4}}$ and define $\widetilde \Sigma = \E X\otimes X$. It holds that
\[
\|\Sigma - \widetilde \Sigma\| \le \kappa(4)^4\|\Sigma\|\sqrt{\frac{\mathbf{r}(\Sigma)}{N}}.
\]
Moreover, $\|\widetilde \Sigma\| \le \|\Sigma\|$ and $\tr(\widetilde \Sigma) \le \tr(\Sigma)$.
\end{proposition}
\begin{proof}
The proof follows from the proof Lemma 2.1 in \cite{mendelson2020robust}. For the second part of the proof we observe that $\Sigma - \widetilde \Sigma$ is positive semi-definite. 
\end{proof}

The result of Proposition \ref{prop:truncdoesnothurt} allows us to focus on estimating the matrix of second moments $\widetilde \Sigma$. Indeed, by the triangle inequality, the estimator $\widehat{\Sigma}_{\eta,\delta}$ satisfies
\[
\|\widehat{\Sigma}_{\eta,\delta} - \Sigma\| \le \|\widehat{\Sigma}_{\eta,\delta} - \widetilde\Sigma\| + \kappa(4)^4\|\Sigma\|\sqrt{\frac{\mathbf{r}(\Sigma)}{N}},
\]
and the last term only affects the multiplicative constant factors in our bound. We proceed by defining the following quantity
\begin{equation*}
    Q_0 = C_{Q_0}\max\left\{\frac{1}{\varepsilon}\|\Sigma\|\sqrt{\frac{\mathbf{r}(\Sigma)+\log(2/\delta)}{N}}, \varepsilon^{-2/p}\kappa(p)^2\|\Sigma\|\right\}.
\end{equation*}
Here $C_{Q_0} = C_{Q_0}(p)>0$ is to be chosen later, depends only on $p$ and is non-increasing with respect to $p$. The first term in the definition above is responsible for the rate in the uncontaminated case and the second term captures the rate of $\eta$ according to the value of $p$. We start with an important lemma that is the second order analog of \cite[Lemma 1]{lugosi2021robust}. For the rest of this section, $\widetilde{\Sigma}$ is the matrix of second moments of the truncated vector $X = Y\ind{\|Y\| \le (N\tr(\Sigma)\|\Sigma\|)^{1/4})}$ and $\varepsilon = \max\left(20\eta, 560\frac{\log(2/\delta)}{N}\right)$.

\begin{lemma}
\label{lem:empirical_quantile}
Let $Y \in \mathbb{R}^d$ be a mean zero random vector satisfying the $L_p-L_2$ norm equivalence assumption with $p>4$. Let $Y_1,\ldots,Y_N$ be i.i.d. copies of $Y$ and set $Z_i=Y_i\ind{\|Y_i\| \le (N\tr(\Sigma)\|\Sigma\|)^{1/4})}$ for every $i\in [N]$, then, with probability at least $1-\delta/2$,
\begin{equation*}
    \sup_{v\in S^{d-1}}\left|\left\{i \in [N]:\langle Z_i,v\rangle^2-v^{\top}\widetilde{\Sigma} v \ge Q_0\right\}\right|\le \frac{\varepsilon}{4}N.
\end{equation*}
\end{lemma}
\begin{proof}
For simplicity of notation, we write $\overline{Z}_i(v)=\langle Z_i,v\rangle^2-v^{\top}\widetilde{\Sigma} v$. The proof is a standard application of a small ball argument in empirical process theory. Consider a function $\xi:\mathbb{R}\rightarrow \mathbb{R}$ defined by 
\begin{equation*}
    \xi(x)= \begin{cases}
      0, &x\le Q_0/2,\\
      \frac{2x}{Q_0} -1, &x\in (Q_0/2,Q_0], \\
      1, &x>Q_0.
    \end{cases}
\end{equation*}

It is clear that $\ind{\overline{Z}_i(v)\ge Q_0}\le \xi(\overline{Z}_i(v))\le\ind{\overline{Z}_i(v)\ge Q_0/2} $ and $\xi$ is a Lipschitz function with constant $2/Q_0$. By symmetrization \cite{gine1984some} and the contraction lemma for Rademacher processes \cite[Theorem 4.4]{ledoux2013probability}, we have
\begin{equation*}
\begin{split}
\mathbb{E}\sup_{v\in S^{d-1}}\frac{1}{N}\sum_{i=1}^N \ind{\overline{Z}_i(v)\ge Q_0} &\le \mathbb{E}\sup_{v\in S^{d-1}}\frac{1}{N}\sum_{i=1}^N \xi(\overline{Z}_i(v))\\
& \le 2\mathbb{E}\sup_{v\in S^{d-1}}\frac{1}{N}\sum_{i=1}^N \left|\varepsilon_i\xi(\overline{Z}_i(v))\right| + \sup_{v\in S^{d-1}}\mathbb{E}\xi(\overline{Z}_i(v))\\
&\le \frac{4}{Q_0}\mathbb{E}\sup_{v\in S^{d-1}}\frac{1}{N}\left|\sum_{i=1}^N \varepsilon_i\overline{Z}_i(v)\right| + \sup_{v\in S^{d-1}}\mathbb{E}\xi(\overline{Z}(v)).
\end{split}
\end{equation*}
Now, we define $C_{BY} = C_{BY}(p)>0$ to be the constant in the conclusion of Corollary \ref{cor:baiyin}. We remark that $C_{BY}$ is at most an absolute constant for any $p \ge c > 4$, where $c$ is another absolute constant. We now apply Corollary \ref{cor:baiyin} and the same arguments as in its proof and obtain
\begin{equation*}
\begin{split}
    \mathbb{E}\sup_{v \in S^{d - 1}}\frac{1}{N}\sum_{i=1}^N \left|\varepsilon_i (\langle Z_i,v\rangle^2- v^{\top}\widetilde{\Sigma}v)\right| &\le \mathbb{E}\sup_{v \in S^{d - 1}}\frac{1}{N}\left|\sum_{i=1}^n \varepsilon_i \langle Z_i,v\rangle^2\right| + \frac{\|\Sigma\|}{N}\mathbb{E}\left|\sum_{i=1}^N \varepsilon_i\right| \\
    &\leq 2C_{BY} \|\Sigma\|\sqrt{\frac{\mathbf{r}(\Sigma)}{N}} \le \frac{Q_0\varepsilon}{128}.
\end{split}
\end{equation*}
The last steps follow from the fact that $\mathbf{r}(\Sigma)\ge 1$ and $\|\widetilde{\Sigma}\|\le \|\Sigma\|$. By the definition of $Q_0$ we conclude that the first term is at most $\frac{\varepsilon}{32}$ if we choose a sufficiently large constant $C_{Q_0}>0$. We now proceed to bound the second term. We use Markov's inequality together with the $L_p-L_2$ norm equivalence to obtain
\begin{equation*}
\begin{split}
\mathbb{E}\xi\left(\langle Z,v\rangle^2-v^{\top}\widetilde{\Sigma} v\right) \le \mathbb{P}\left(\langle Z,v\rangle^2\ge Q/2+v^{\top}\widetilde{\Sigma} v\right) &\leq \frac{\kappa(p)^p}{\left(Q/2+v^{\top}\widetilde{\Sigma} v\right)^{p/2}}(\mathbb{E}\langle Z,v\rangle^2)^{p/2}\\
&\leq \left(\frac{2\kappa(p)^2\|\Sigma\|}{Q}\right)^{p/2}.
\end{split}
\end{equation*}
Again, by the definition of $Q_0$, we can choose $C_{Q_0}>0$ such that the right hand side above is at most $\frac{\varepsilon}{32}$. By Talagrand's concentration inequality for supremum of empirical process (Massart's version \cite{massart2000constants}), with probability at least $1-\exp(-x)$,
\begin{equation*}
\sup_{v\in S^{d-1}}\frac{1}{N}\sum_{i=1}^N \ind{\overline{Z}_i(v)\ge Q} \le \frac{\varepsilon}{8} + \sqrt{\frac{8x}{N}}\frac{\sqrt{\varepsilon}}{8} + \frac{35 x}{N}.
\end{equation*}
We choose $x = \varepsilon N/560$ to conclude the proof.
\end{proof}
\noindent  We now describe the first part of the estimation procedure. First, we assume without loss of generality that $\varepsilon < 1$. When the conclusion of the previous lemma holds, we immediately obtain that, for every $v\in S^{d-1}$ and $Q\in (2Q_0,4Q_0)$,
\begin{equation*}
\quad q_v -v^{\top}\widetilde{\Sigma} v + Q \le 5Q_0 \quad \text{and} \quad q_v -v^{\top}\widetilde{\Sigma} v + Q \ge Q_0.
\end{equation*}
To see why the latter condition holds, suppose by contradiction that $q_v -v^{\top}\widetilde{\Sigma} v + Q \le Q_0$, then $q_v \le v^{\top}\widetilde{\Sigma} v -Q_0$. Since $\varepsilon < 1$, by the choice of $Q_0$, we have $Q_0 > \|\Sigma\|$ and then we get $0\le q_v \le v^{\top}\widetilde{\Sigma} v -Q_0 <0$, the latter is clearly a contradiction. Moreover, both inequalities above still hold for an $\eta$-corrupted sample, since $\varepsilon\ge 20\eta$ and by the definition of $\varepsilon$, we have that the fraction of points greater than $Q_0$ is at most $\frac{\varepsilon}{4}+\eta \le \varepsilon(\frac{1}{20}+\frac{1}{4})<\frac{\varepsilon}{2}$. Thus, we can use $\widetilde{Z}_1,\ldots,\widetilde{Z}_N$ to estimate $q_v$, and since they are independent from the second half of the sample, from now on we work conditionally on the event of Lemma \ref{lem:empirical_quantile}.

The second part of the analysis consists in proving that $\Gamma(v,Q)$ is non-empty. The formal statement is the proposition below.

\begin{proposition}
\label{prop:U_Q(v)}
Under the notation of Theorem \ref{thm:thecasep_greater_four}, fix $Q \in [2Q_0,4Q_0]$ and assume that the event of Lemma \ref{lem:empirical_quantile} holds. There exists an absolute constant $C>0$ such that if $C_{Q_0}\ge C$, then, with probability at least $1-\delta/2$, 
\begin{equation*}
    \sup_{v\in S^{d-1}} \left|\frac{1}{\lambda_v(Q)N}\sum_{i=1}^N \psi\left(\lambda_v(Q)\langle \widetilde{X}_i,v\rangle^2\right) - v^{\top}\widetilde{\Sigma} v\right| \le 4\varepsilon Q.
\end{equation*}
\end{proposition}
It will be convenient to work with two a sided truncation function in the analysis. We define the following function for arbitrary positive $\lambda_1 \ge \lambda_2$,
\begin{equation*}
    \psi_{\lambda_1,\lambda_2}(x)=\begin{cases}
      \frac{1}{\lambda_2}, \quad x>\frac{1}{\lambda_2},\\
      x, \quad x\in [\frac{1}{\lambda_1},\frac{1}{\lambda_2}],\\
      \frac{1}{\lambda_1}, \quad x<\frac{1}{\lambda_1}.
    \end{cases}
\end{equation*}
Clearly, if $\lambda_1=\lambda_2=\lambda$, our definition agrees with \eqref{eq:truncfunction}. That is,  $\frac{1}{\lambda}\psi(\lambda x) = \psi_{\lambda,\lambda}(x)$.
\begin{proof}
First, for every $v\in S^{d-1}$, due to the assumption that the event of Lemma \ref{lem:empirical_quantile} holds, the uncorrupted observations satisfy 
\begin{equation*}
\left| \frac{1}{\lambda_v(Q)N}\sum_{i=1}^N \psi\left(\lambda_v(Q)\langle \widetilde{X}_i,v\rangle^2\right) - \frac{1}{\lambda_v(Q)N}\sum_{i=1}^N \psi\left(\lambda_v(Q)\langle X_i,v\rangle^2\right)\right| \le 2\eta Q \le \frac{\varepsilon Q}{10}.
\end{equation*}
We centralize the empirical process to obtain that
\begin{equation*}
\frac{1}{\lambda_v(Q)N}\sum_{i=1}^N \psi\left(\lambda_v(Q)\langle X_i,v\rangle^2\right) -v^{\top}\widetilde{\Sigma} v = \frac{1}{\lambda_v^{c}(Q)N}\sum_{i=1}^N \psi\left(\lambda_v^c(Q)(\langle X_i,v\rangle^2-v^{\top}\widetilde{\Sigma} v)\right),
\end{equation*}
where $\lambda_v^c(Q)=(q_v+Q-v^{\top}\widetilde{\Sigma} v)^{-1}$. To see why this identity holds, observe that both sides are equal except when $\langle X_i,v\rangle^2 - v^{\top}\widetilde{\Sigma} v \le -q_v - Q+v^{\top}\widetilde{\Sigma} v$. In this case, $0\le \langle X_i,v\rangle^2 \le -q_v -Q + 2v^{\top}\widetilde{\Sigma} v$. However, since $Q\ge 2\|\Sigma\|$ if $C_{Q_0}\ge 2$ and $\varepsilon < 1$, the latter inequality cannot be satisfied. Now, on the event of Lemma \ref{lem:empirical_quantile}, $Q_0\le q_v -v^{\top}\widetilde{\Sigma} v  + Q\le 5Q_0 $ and therefore
\begin{equation*}
\frac{1}{\lambda_v^{c}(Q)N}\sum_{i=1}^N \psi\left(\lambda_v^c(Q)(\langle X_i,v\rangle^2-v^{\top}\widetilde{\Sigma} v)\right) \le \frac{1}{N}\sum_{i=1}^N\psi_{-\frac{1}{Q_0},\frac{1}{5Q_0}}\left(\langle X_i,v\rangle^2-v^{\top}\widetilde{\Sigma} v\right).
\end{equation*}
We define $\overline{U}_Q(v)$ to be the right hand side in the inequality above. We first study how the empirical process of interest concentrates around the mean. To simplify the analysis, we consider a simpler quantity, $\overline{W}_Q(v)=\frac{1}{N}\sum_{i=1}^N \psi_{\frac{-1}{3Q},\frac{1}{3Q}}\left(\langle X_i,v\rangle^2 -v^{\top}\widetilde{\Sigma}v\right)$. Our starting point is
\begin{align*}
\sup_{v\in S^{d-1}}(\overline{U}_Q(v)- \mathbb{E}\overline{U}_Q(v)) &\le \sup_{v\in S^{d-1}}(\overline{U}_Q(v)- \overline{W}_Q(v)) 
 + \sup_{v\in S^{d-1}}(\overline{W}_Q(v)- \mathbb{E}\overline{W}_Q(v))
 \\
 &\qquad+ \sup_{v\in S^{d-1}}(\mathbb{E}\overline{W}_Q(v)- \mathbb{E}\overline{U}_Q(v))\\
&= (a) + (b) + (c).
\end{align*}
For the first term, observe that it is non-zero if and only if the argument of the function is above $5Q_0$ or below $-Q_0$. The latter cannot happen because of the same contradiction argument as above. The magnitude of the difference is at most $3Q$. By Lemma \ref{lem:empirical_quantile}, there are at most $\varepsilon N/4$ points in the range $(5Q_0,\infty)$, therefore, we obtain that the first term (a) is at most $\frac{3\varepsilon Q}{4}$. Similarly, we apply the Markov inequality together with the $L_p-L_2$ norm equivalence assumption to obtain that
\begin{equation*}
\begin{split}
\sup_{v\in S^{d-1}}\mathbb{E}\left(\overline{W}_Q(v) - \overline{U}_Q(v)\right) &\le 3Q\mathbb{P}(\langle X,v\rangle^2 - v^{\top}\widetilde{\Sigma} v >5Q_0)\\
&\leq 3Q \left(\frac{\kappa(p)^2\|\Sigma\|}{5Q_0}\right)^{p/2} \le \frac{3\varepsilon Q}{32}.
\end{split}
\end{equation*}
This concludes the bound for (c). The second term (b) is handled via Talagrand's concentration inequality for empirical processes (Massart's version \cite{massart2000constants}). Observe that, for every $v\in S^{d-1}$,
\begin{equation*}
\left|\psi_{-\frac{1}{3Q},\frac{1}{3Q}}\left(\langle X,v\rangle^2 - v^{\top}\widetilde{\Sigma} v\right)\right| \le 3Q \quad \text{and} \quad \mathbb{E}\left|\psi_{-\frac{1}{3Q},\frac{1}{3Q}}\left(\langle X,v\rangle ^2\right)\right|^2 \le \mathbb{E}\langle X,v\rangle^4 \leq \kappa(4)^4 \|\Sigma\|^2.
\end{equation*}
Moreover, for every $Q>0$, $\psi(\cdot)_{-\frac{1}{Q},\frac{1}{Q}}$ is a $1$-Lipschitz function that passes through the origin. By the same symmetrization and contraction arguments as in the proof of Lemma \ref{lem:empirical_quantile}, we obtain that
\begin{equation*}
\mathbb{E}\sup_{v\in S^{d-1}}\left|\overline{W}_Q(v) - \mathbb{E}\overline{W}_Q(v)\right| \le 2 \mathbb{E}\sup_{v\in S^{d-1}}\left|\sum_{i=1}^N \varepsilon_i \langle X_i,v\rangle^2\right| \leq 2C_{BY} \|\Sigma\| \sqrt{\frac{\mathbf{r}(\Sigma)}{N}}.
\end{equation*}
Therefore, by Talagrand's concentration inequality for empirical processes, with probability at least $1-2e^{-x}$, it holds that
\begin{equation*}
    \sup_{v\in S^{d-1}}|W_Q(v)-\mathbb{E}W_Q(v)| \leq 4C_{BY}\|\Sigma\|\sqrt{\frac{\mathbf{r}(\Sigma)}{N}} +\kappa(4)^2\|\Sigma\|\sqrt{\frac{8x}{N}} + 105Q \frac{x}{N}.
\end{equation*}
We choose $x = 2\log (2/\delta)$ to obtain that the right hand side above is less than $\frac{\varepsilon Q}{2}$. Now, the problem boils down to control uniformly the typical magnitude of $\overline{U}_Q(v)$. For simplicity, we define $\overline{X}_i(v)=\langle X_i,v\rangle^2-v^{\top}\widetilde{\Sigma}v$ and write
\begin{equation*}
\begin{split}
\sup_{v\in S^{d-1}}\mathbb{E}\overline{U}_Q(v)
&= \sup_{v\in S^{d-1}}\mathbb{E}[\psi_{-\frac{1}{Q_0},\frac{1}{5Q_0}}(\langle X,v\rangle^2 - v^{\top}\widetilde{\Sigma} v))(\ind{\overline{X}_i(v)\le 5Q_0}+\ind{\overline{X}_i(v)> 5Q_0})]\\
&\le \sup_{v\in S^{d-1}}5Q_0\mathbb{E}\ind{\overline{X}(v)\ge 5Q_0}\\
&\le 5Q\frac{\varepsilon}{64}.
\end{split}
\end{equation*}
The final bound for the first part becomes $\varepsilon Q(\frac{1}{10}+\frac{3}{4}+\frac{3}{32}+\frac{5}{64}+\frac{1}{2})\le 2\varepsilon Q$. For the other side bound, observe that 
\begin{equation*}
\begin{split}
v^{\top}\widetilde{\Sigma} v-\frac{1}{\lambda_v(Q)N}\sum_{i=1}^N\psi(\lambda_v(Q)\langle X_i,v\rangle^2) &= \frac{1}{\lambda_v^c(Q)N}\sum_{i=1}^N\psi\left(\lambda_v^c(Q)v^{\top}\widetilde{\Sigma} v- \langle X_i,v\rangle^2\right)\\
&\le \frac{1}{N}\sum_{i=1}^N\psi_{-\frac{1}{Q_0},\frac{1}{5Q_0}}\left(v^{\top}\widetilde{\Sigma} v-\langle X_i,v\rangle^2\right).
\end{split}
\end{equation*}
The same analysis follows and concludes the proof.
\end{proof}

The proof of Theorem \ref{thm:thecasep_greater_four} follows directly from Lemma \ref{lem:empirical_quantile} and Proposition \ref{prop:U_Q(v)}. Indeed, consider the event $\mathcal{E}$ where the events of both Lemma \ref{lem:empirical_quantile} and Proposition \ref{prop:U_Q(v)} hold. It occurs with probability at least $1-\delta$ as desired. Our next argument is standard in the literature \cite{lugosi2021robust} and is similar to Lepskii's method \cite{lepskii1992asymptotically}. Let $i_0$ be the smallest integer satisfying $2^{i_0} \in (2Q_0,4Q_0)$. By Lemma \ref{lem:empirical_quantile} we know that at most $\frac{\varepsilon N}{4}$ samples are outside of the range $q_v + 2^{i_0}$. We conclude that the the difference between $U_Q(v)$ and $U_{2Q}(v)$ is at most $\frac{\varepsilon Q}{2}$. By induction, this implies that $\Gamma(2^{i_0}) \subset \Gamma(2^{i_0+1})$ for every $i\ge i_0$. This implies that corresponding sets are nested. The diameter of the set $\bigcap_{i\ge i^{\ast}}\Gamma(2^{i})$ is at most $8\varepsilon Q$ and the matrix of second moments $\widetilde{\Sigma}$ belongs to an even smaller set. Thus, $\|\hat{\Sigma}-\widetilde{\Sigma}\|\le 8\varepsilon Q$. We conclude by applying Proposition \ref{prop:truncdoesnothurt}. \qed

\section{Optimality}
\label{sec:optimality}
It is natural to ask about the optimality of our statistical guarantees of Theorem \ref{thm:informalmain}. It is well known that the term $C\|\Sigma\| \sqrt{\mathbf{r}(\Sigma)/N}$ appears in the minimax lower bound for covariance estimation  \cite[Theorem 2]{lounici2014high} as well as in the lower bound for the performance of the sample covariance matrix in the Gaussian case \cite[Theorem 4]{koltchinskii2017operators}. Clearly, this term is necessary because the adversary can always leave the sample uncorrupted and the multivariate Gaussian distribution satisfies the norm equivalence assumption \eqref{eq:momeqv}. Further, we show that the term $C\|\Sigma\| \sqrt{\log(1/\delta)/N}$ cannot be improved in general. As shown in \cite{mendelson2020robust} (following the lower bound in \cite{lugosi2019near}), the term scaling as $R\sqrt{\log(1/\delta)/N}$ appears in the lower bound for any estimator of the covariance matrix. Here $R^2$ is the so-called \emph{weak variance} defined by 
\[
R^2 = \sup\nolimits_{v \in S^{d - 1}}\E\left(v^{\top}(X \otimes X - \E X \otimes X)v\right)^2.
\]
In the multivariate Gaussian case, using the standard relation between moments, we get
\[
R^2 \ge \sup\nolimits_{v \in S^{d - 1}}\E\langle v, X\rangle^4 - \|\Sigma\|^2 = 3\sup\limits_{v \in S^{d - 1}}(\E\langle v, X\rangle)^2 - \|\Sigma\|^2 = 2\|\Sigma\|^2.
\]
This shows the necessity of the term $C\|\Sigma\| \sqrt{\log(1/\delta)/N}$ in Theorem \ref{thm:informalmain}.

Finally, we can restrict the discussion to the convergence rate with respect to the fraction of corrupted samples $\eta$. Some closely related lower bounds with matching rates of convergence appear in the analysis of the sparse vector model \cite[Theorem 5 and Theorem 6]{comminges2021adaptive}. Their model can be seen as a special case of our setup when $d = 1$. Instead of exploiting their techniques for the lower bounds, we provide a separate analysis using explicit examples. 
Our key argument to derive the minimax optimality with respect to $\eta$ is to reduce our problem to a mean estimation problem. For the mean estimation, a simple computation shows a lower bound with respect to $\eta$. To describe this result, we first consider a basic definition.
\begin{definition}
For a random variable $X$, we define the quantile
\begin{equation*}
    Q_q(X) = \sup\{M\in \mathbb{R}: \mathbb{P}(X \ge M) \ge 1-q\}.
\end{equation*}
\end{definition}
Formally, the following simple lower bound holds.
\begin{proposition}[Inequality 2.3 in \cite{lugosi2021robust}]
Let $X$ be a random variable with mean $\mu$, variance $\sigma_X$, and with an absolutely continuous distribution. Suppose that  $\widetilde{X}_1,\ldots,\widetilde{X}_N$ is an $\eta$-corrupted sample sampled according to the distribution of $X$.
Let $\overline{X}=X-\mathbb{E}X$ and define $\epsilon(\overline{X},\eta)$ as follows
\begin{equation*}
\epsilon (\overline{X},\eta)= \max\left\{\mathbb{E}|\overline{X}-Q_{\eta/2}(\overline{X})|\ind{\overline{X}\le Q_{\eta/2}(\overline{X})},\mathbb{E}|\overline{X}-Q_{1-\eta/2}(\overline{X})|\ind{\overline{X}\ge Q_{1-\eta/2}(\overline{X})}\right\}.
\end{equation*}
Then, no estimator $\widehat{\mu}=\widehat{\mu}(\widetilde{X}_1,\ldots,\widetilde{X}_N)$ of the mean $\mu$  can perform better than
\begin{equation*}
    |\widehat{\mu} - \mu| \le \epsilon(\overline{X},\eta).
\end{equation*}
\end{proposition}

To obtain the minimax rates with respect to $\eta$, it is enough to restrict our attention to the one dimensional case. Assume that the distribution of a zero mean random variable $X$ satisfies the $L_4-L_2$ norm equivalence \eqref{eq:momeqv}. That is, for some $\kappa \ge 1$, it holds that $\left(\E X^4\right)^{1/4} \le \kappa \left(\E X^2\right)^{1/2}$. Denote $Y = X^2$ and observe that $Y$ is non-negative. In this case, the estimation of the variance of $X$ can be seen as the estimation of the mean of $Y$. The norm equivalence assumption can be rewritten as
$
    \left(\E Y^2\right)^{1/2} \le \kappa^2\E Y. 
$

\begin{example}[Optimality of the $\sqrt{\eta}$-term when $p = 4$]
\label{ex:sqrteta}
First, a few intermediate random variables are considered before the definition of $Y$. Define
\begin{equation*}
    Y_1=\begin{cases}
    -\frac{1}{\sqrt{\eta}}, &\text{with probability} \  \frac{\eta}{2},\\
    -1, &\text{with probability} \frac{1-\eta}{2},\\
    1, &\text{with probability} \  \frac{1-\eta}{2},\\
    \frac{1}{\sqrt{\eta}}, &\text{with probability} \  \frac{\eta}{2}.
    \end{cases}
\end{equation*}
Clearly $\mathbb{E}Y_1 =0$. Let $\sigma_{Y_1}$ denote the standard deviation of $Y_1$. It holds that
\begin{equation*}
\sigma_{Y_1}=(\mathbb{E}Y_1^2)^{1/2} = \sqrt{2-\eta} \le \sqrt{2}.
\end{equation*}
Assume that $\eta \le 1/4$. We have $Q_{\eta/2}(Y_1) = -1$. Indeed, $\mathbb{P}(Y_1\ge -1) = 1-\frac{\eta}{2}$ and $\mathbb{P}(Y_1\ge 1) = \frac{1}{2}< 1-\frac{\eta}{2}$. We conclude that
\begin{equation*}
\epsilon(Y_1,\eta)\ge \mathbb{E}|Y_1+1|\ind{Y_1\le -1} = \left(\frac{1}{\sqrt{\eta}}-1\right)\frac{\eta}{2} = \frac{\sqrt{\eta}}{2}-\frac{\eta}{2}\ge \frac{1}{4}\sqrt{\eta}.
\end{equation*}
We now consider the random variable $Y_2=\frac{\sigma^2}{\sqrt{2-\eta}}Y_1$, where $\sigma > 0$ is a positive real number. It still holds that $\mathbb{E}Y_2 = 0$, but now, $(\mathbb{E}Y_2^2)^{1/2} = \sigma^2$. By homogeneity, we have 
\[
\epsilon(Y_2,\eta)\ge \frac{\sigma^2\sqrt{\eta}}{4\sqrt{2-\eta}} \ge \frac{\sigma^2\sqrt{\eta}}{4\sqrt{2}}.
\]
Finally, observe that $Y=Y_2 +\|Y_2\|_{\infty}$ is a non-negative random variable and $\epsilon(Y,\eta) = \epsilon(Y_2,\eta)$ since the difference between $Y$ and $Y_2$ is a constant. We conclude by observing that $\E Y = \|Y_2\|_{\infty}$ and $\E Y^2 = \E Y_2^2 + \|Y_2\|^2_{\infty} \le 2\|Y_2\|^2_{\infty}$. Thus, the norm equivalence assumption $\left(\E Y^2\right)^{1/2} \le \kappa^2\E Y$ holds with $\kappa^2 = \sqrt{2}$.
\end{example}

We now present an example that attains the sub-exponential minimax rate in the mean estimation and therefore confirms the optimality of our covariance estimation results in the sub-Gaussian case. Indeed, observe that if $X$ is a sub-Gaussian random variable, then $Y = X^2$ is a sub-exponential random variable (see \cite[Lemma 2.7.6]{Vershynin2016HDP}). Therefore, when making a reduction from covariance estimation to a mean estimation problem as in Example \ref{ex:sqrteta}, we have to analyze the mean estimation problem in the sub-exponential regime. We note that a similar sub-Gaussian construction appears in \cite[Remark before Section 2.1]{lugosi2021robust}\footnote{To make their construction work in the Gaussian case, we changed the definition of $Q$, so that it becomes an $\eta/4$-quantile instead of an $\eta/2$-quantile claimed in \cite{lugosi2021robust}.}.

\begin{example}[Optimality of $\eta\log(1/\eta)$-term in sub-Gaussian covariance estimation]
Let $Y= \min\{1,|g|^2\}\ind{|g|^{2}< Q} + |g|^2\ind{|g|^2\ge Q}$, where $g$ is a standard Gaussian random variable and $Q$ is such that $\mathbb{P}(|g|^2 \ge Q) = \eta/4$. Consequently, we have $\mathbb{P}(|g|\ge \sqrt{Q}) = \eta/4$. It is easy to see that for $\eta$ sufficiently small, it holds  that $Q > 1$. Let us now prove that $Q_{1 - \eta/2}(Y) = 1$. Indeed, since $\Pr(Y \ge 1) = \Pr(|g| \ge 1) > \eta/2$ for small enough $\eta$, then $Q_{1 - \eta/2}(Y) \ge 1$. It is clear that if $Q_{1 - \eta/2}(Y) > 1$, then $Q_{1 - \eta/2}(Y) \ge Q$, which contradicts  $\mathbb{P}(|g|^2 \ge Q) = \eta/4 < \eta/2$.
Thus, we have
\[
\varepsilon(\overline{Y},\eta) \ge \mathbb{E}|\overline{Y}-Q_{1-\eta/2}(\overline{Y})|\ind{\overline{Y}\ge Q_{1-\eta/2}(\overline{Y})} = \mathbb{E}(Y-1)\ind{Y\ge 1}.
\]
This leads to
\begin{align*}
    \mathbb{E}(Y-1)\ind{Y\ge 1} &= \mathbb{E}Y\ind{1 \le Y < Q} + \mathbb{E}Y\ind{Y\ge Q} - \Pr(Y \ge 1)
    \\
    &=\Pr(1 \le |g|^2 < Q) + \mathbb{E}|g|^2\ind{|g|^2\ge Q} - \Pr(|g|^2 \ge 1)
    \\
    &= \mathbb{E}|g|^2\ind{|g|^2\ge Q} - \Pr(|g|^2 \ge Q).
\end{align*}
Using the standard computation (see \cite[Exercise 2.1.4]{Vershynin2016HDP}), we obtain
\[
\mathbb{E}|g|^2\ind{|g|^2\ge Q} = \frac{2}{\sqrt{2\pi}}\sqrt{Q}\exp(-Q/2) + \Pr(|g| \ge \sqrt{Q}).
\]
Therefore, since $\eta/4= \mathbb{P}(|g|\ge \sqrt{Q})$, and by the standard Gaussian integration \cite[Proposition 2.1.2]{Vershynin2016HDP} we have
\[
\varepsilon(\overline{Y},\eta) \ge \frac{2}{\sqrt{2\pi}}\sqrt{Q}\exp(-Q/2) \ge Q\Pr(|g| \ge \sqrt{Q}) = \frac{Q\eta}{4}.
\]
Using the same Gaussian integration formula, we conclude that, in order to get $\Pr(|g| \ge \sqrt{Q}) = \eta/4$, we need to choose $Q\sim \log\left(1/\eta\right)$. This implies that 
\[
\epsilon(\overline{Y},\eta)\ge \frac{Q\eta}{4} \gtrsim \eta \log\left(1/\eta\right).
\]

We claim that the rate $\eta\log\left(1/\eta\right)$ is sharp. To do so, we compute the upper bound with respect to $\eta$ in Theorem \ref{thm:informalmain}. If the original random variable $X$ is a sub-Gaussian, then $\kappa(p)^2 \sim p$ for every $p\ge 1$ (see \cite[Proposition 2.5.2]{Vershynin2016HDP}) and the upper bound scales as $\eta p/\eta^{2/p}$. We choose $p\sim \log\left(1/\eta\right)$ to optimize the fraction and obtain the claimed rate.
\end{example}
Finally, we note that for the Gaussian case, it is possible to achieve a slightly improved dependence on $\eta$ by eliminating the logarithmic factor $\log(1/\eta)$ (as shown in the follow-up work of Minasyan and the second-named author of this paper \cite{minasyan2023statistically}), owing to the rotational invariance property of the Gaussian distribution. As discussed in \cite{minasyan2023statistically}, the estimator needs to be tailored to the Gaussian distribution and therefore it cannot be applied to the more general setup considered in our Theorem \ref{thm:informalmain}.

\paragraph{Acknowledgments.}  The authors would like to thank Afonso Bandeira for several valuable discussions, and Alexander Tsybakov for pointing out the connection with \cite{comminges2021adaptive}. Nikita Zhivotovskiy is funded in part by ETH Foundations of Data Science (ETH-FDS).

\bibliographystyle{abbrv}  
{\footnotesize
\bibliography{mybib}

\begin{thebibliography}{10}

\bibitem{adamczak2010quantitative}
R.~Adamczak, A.~Litvak, A.~Pajor, and N.~Tomczak-Jaegermann.
\newblock Quantitative estimates of the convergence of the empirical covariance
  matrix in log-concave ensembles.
\newblock {\em Journal of the American Mathematical Society}, 23(2):535--561,
  2010.

\bibitem{appert2021new}
G.~Appert and O.~Catoni.
\newblock New bounds for $ k $-means and information $ k $-means.
\newblock {\em arXiv preprint arXiv:2101.05728}, 2021.

\bibitem{audibert2011robust}
J.-Y. Audibert and O.~Catoni.
\newblock Robust linear least squares regression.
\newblock {\em The Annals of Statistics}, 39(5):2766--2794, 2011.

\bibitem{baiyin1994}
Z.~D. Bai and Y.~Q. Yin.
\newblock {Limit of the smallest eigenvalue of a large dimensional sample
  covariance matrix}.
\newblock {\em The Annals of Probability}, 21(3):1275 -- 1294, 1993.

\bibitem{bickel1965some}
P.~J. Bickel.
\newblock On some robust estimates of location.
\newblock {\em The Annals of Mathematical Statistics}, pages 847--858, 1965.

\bibitem{bickel2008covariance}
P.~J. Bickel and E.~Levina.
\newblock Covariance regularization by thresholding.
\newblock {\em The Annals of Statistics}, 36(6):2577--2604, 2008.

\bibitem{bickel2008regularized}
P.~J. Bickel and E.~Levina.
\newblock Regularized estimation of large covariance matrices.
\newblock {\em The Annals of Statistics}, 36(1):199--227, 2008.

\bibitem{bourgain1996random}
J.~Bourgain.
\newblock Random points in isotropic convex sets.
\newblock {\em Convex Geometric Analysis}, 34:53--58, 1996.

\bibitem{brownlees2015empirical}
C.~Brownlees, E.~Joly, and G.~Lugosi.
\newblock Empirical risk minimization for heavy-tailed losses.
\newblock {\em The Annals of Statistics}, 43(6):2507--2536, 2015.

\bibitem{cai2013optimal}
T.~T. Cai, Z.~Ren, and H.~H. Zhou.
\newblock Optimal rates of convergence for estimating toeplitz covariance
  matrices.
\newblock {\em Probability Theory and Related Fields}, 156(1):101--143, 2013.

\bibitem{cai2016estimating}
T.~T. Cai, Z.~Ren, and H.~H. Zhou.
\newblock Estimating structured high-dimensional covariance and precision
  matrices: Optimal rates and adaptive estimation.
\newblock {\em Electronic Journal of Statistics}, 10(1):1--59, 2016.

\bibitem{cai2010optimal}
T.~T. Cai, C.-H. Zhang, and H.~H. Zhou.
\newblock Optimal rates of convergence for covariance matrix estimation.
\newblock {\em The Annals of Statistics}, 38(4):2118--2144, 2010.

\bibitem{catoni2012challenging}
O.~Catoni.
\newblock Challenging the empirical mean and empirical variance: a deviation
  study.
\newblock {\em Annales de l'{Institut Henri Poincar{\'e}}, Probabilit{\'e}s et
  Statistiques}, 48(4):1148--1185, 2012.

\bibitem{catoni2016pac}
O.~Catoni.
\newblock {PAC-Bayesian} bounds for the {Gram} matrix and least squares
  regression with a random design.
\newblock {\em arXiv preprint arXiv:1603.05229}, 2016.

\bibitem{catoni2017dimension}
O.~Catoni and I.~Giulini.
\newblock Dimension-free {PAC-Bayesian} bounds for matrices, vectors, and
  linear least squares regression.
\newblock {\em arXiv preprint arXiv:1712.02747}, 2017.

\bibitem{chen2018robust}
M.~Chen, C.~Gao, and Z.~Ren.
\newblock Robust covariance and scatter matrix estimation under {H}uber’s
  contamination model.
\newblock {\em The Annals of Statistics}, 46(5):1932--1960, 2018.

\bibitem{cheng2019faster}
Y.~Cheng, I.~Diakonikolas, R.~Ge, and D.~P. Woodruff.
\newblock Faster algorithms for high-dimensional robust covariance estimation.
\newblock In {\em Conference on Learning Theory}, pages 727--757. PMLR, 2019.

\bibitem{cherapanamjeri2020algorithms}
Y.~Cherapanamjeri, S.~B. Hopkins, T.~Kathuria, P.~Raghavendra, and
  N.~Tripuraneni.
\newblock Algorithms for heavy-tailed statistics: Regression, covariance
  estimation, and beyond.
\newblock In {\em Proceedings of the 52nd Annual ACM SIGACT Symposium on Theory
  of Computing}, pages 601--609, 2020.

\bibitem{chinot2019robust}
G.~Chinot, G.~Lecu{\'e}, and M.~Lerasle.
\newblock Robust statistical learning with {L}ipschitz and convex loss
  functions.
\newblock {\em Probability Theory and Related Fields}, pages 1--44, 2019.

\bibitem{comminges2021adaptive}
L.~Comminges, O.~Collier, M.~Ndaoud, and A.~B. Tsybakov.
\newblock Adaptive robust estimation in sparse vector model.
\newblock {\em The Annals of Statistics}, 49(3):1347--1377, 2021.

\bibitem{depersin2020robust}
J.~Depersin.
\newblock Robust subgaussian estimation with {VC}-dimension.
\newblock {\em arXiv preprint arXiv:2004.11734}, 2020.

\bibitem{diakonikolas2019recent}
I.~Diakonikolas and D.~M. Kane.
\newblock Recent advances in algorithmic high-dimensional robust statistics.
\newblock {\em arXiv preprint arXiv:1911.05911}, 2019.

\bibitem{diakonikolas2022outlier}
I.~Diakonikolas, D.~M. Kane, J.~C. Lee, and A.~Pensia.
\newblock Outlier-robust sparse mean estimation for heavy-tailed distributions.
\newblock {\em arXiv preprint arXiv:2211.16333}, 2022.

\bibitem{el2008operator}
N.~El~Karoui.
\newblock Operator norm consistent estimation of large-dimensional sparse
  covariance matrices.
\newblock {\em The Annals of Statistics}, 36(6):2717--2756, 2008.

\bibitem{el2018random}
N.~El~Karoui.
\newblock Random matrices and high-dimensional statistics: {B}eyond covariance
  matrices.
\newblock In {\em Proceedings of the International Congress of Mathematicians:
  Rio de Janeiro 2018}, pages 2857--2876. World Scientific, 2018.

\bibitem{gine1984some}
E.~Gin{\'e} and J.~Zinn.
\newblock Some limit theorems for empirical processes.
\newblock {\em The Annals of Probability}, pages 929--989, 1984.

\bibitem{giulini2017robust}
I.~Giulini.
\newblock Robust {PCA} and pairs of projections in a {H}ilbert space.
\newblock {\em Electronic Journal of Statistics}, 11(2):3903--3926, 2017.

\bibitem{giulini2018robust}
I.~Giulini.
\newblock Robust dimension-free {G}ram operator estimates.
\newblock {\em Bernoulli}, 24(4B):3864--3923, 2018.

\bibitem{guedon2017interval}
O.~Gu{\'e}don, A.~E. Litvak, A.~Pajor, and N.~Tomczak-Jaegermann.
\newblock On the interval of fluctuation of the singular values of random
  matrices.
\newblock {\em Journal of the European Mathematical Society}, 19(5):1469--1505,
  2017.

\bibitem{hampel1971general}
F.~R. Hampel.
\newblock A general qualitative definition of robustness.
\newblock {\em The Annals of Mathematical Statistics}, 42(6):1887--1896, 1971.

\bibitem{hardle2021robustifying}
W.~H{\"a}rdle, Y.~Klochkov, A.~Petukhina, and N.~Zhivotovskiy.
\newblock Robustifying {M}arkowitz.
\newblock Technical report, IRTG 1792 Discussion Paper, 2021.

\bibitem{hsu2016heavy}
D.~Hsu and S.~Sabato.
\newblock Loss minimization and parameter estimation with heavy tails.
\newblock {\em Journal of Machine Learning Research}, 17(18):1--40, 2016.

\bibitem{huber1964}
P.~J. Huber.
\newblock {Robust Estimation of a Location Parameter}.
\newblock {\em The Annals of Mathematical Statistics}, 35(1):73 -- 101, 1964.

\bibitem{huber1981robust}
P.~J. Huber.
\newblock Robust {S}tatistics.
\newblock {\em Wiley Series in Probability and Mathematical Statistics}, 1981.

\bibitem{kannan1997random}
R.~Kannan, L.~Lov{\'a}sz, and M.~Simonovits.
\newblock Random walks and an ${O}^*(n^5)$ volume algorithm for convex bodies.
\newblock {\em Random Structures \& Algorithms}, 11(1):1--50, 1997.

\bibitem{ke2019user}
Y.~Ke, S.~Minsker, Z.~Ren, Q.~Sun, and W.-X. Zhou.
\newblock User-friendly covariance estimation for heavy-tailed distributions.
\newblock {\em Statistical Science}, 34(3):454--471, 2019.

\bibitem{klochkov2021robust}
Y.~Klochkov, A.~Kroshnin, and N.~Zhivotovskiy.
\newblock Robust k-means clustering for distributions with two moments.
\newblock {\em The Annals of Statistics}, 49(4):2206--2230, 2021.

\bibitem{koltchinskii2011oracle}
V.~Koltchinskii.
\newblock {\em Oracle Inequalities in Empirical Risk Minimization and Sparse
  Recovery Problems: Ecole d'Et{\'e} de Probabilit{\'e}s de Saint-Flour
  XXXVIII-2008}, volume 2033.
\newblock Springer Science \& Business Media, 2011.

\bibitem{koltchinskii2020asymptotically}
V.~Koltchinskii.
\newblock Asymptotically efficient estimation of smooth functionals of
  covariance operators.
\newblock {\em Journal of the European Mathematical Society}, 23(3):765--843,
  2020.

\bibitem{koltchinskii2017operators}
V.~Koltchinskii and K.~Lounici.
\newblock Concentration inequalities and moment bounds for sample covariance
  operators.
\newblock {\em Bernoulli}, 23(1):110--133, 2017.

\bibitem{lam2009sparsistency}
C.~Lam and J.~Fan.
\newblock Sparsistency and rates of convergence in large covariance matrix
  estimation.
\newblock {\em Annals of statistics}, 37(6B):4254, 2009.

\bibitem{ledoux2013probability}
M.~Ledoux and M.~Talagrand.
\newblock {\em Probability in Banach Spaces: {I}soperimetry and Processes}.
\newblock Springer Science \& Business Media, 2013.

\bibitem{lepskii1992asymptotically}
O.~Lepskii.
\newblock Asymptotically minimax adaptive estimation. {I}: {U}pper bounds.
  {O}ptimally adaptive estimates.
\newblock {\em Theory of Probability \& Its Applications}, 36(4):682--697,
  1992.

\bibitem{lounici2014high}
K.~Lounici.
\newblock High-dimensional covariance matrix estimation with missing
  observations.
\newblock {\em Bernoulli}, 20(3):1029--1058, 2014.

\bibitem{lugosi2019mean}
G.~Lugosi and S.~Mendelson.
\newblock Mean estimation and regression under heavy-tailed distributions: A
  survey.
\newblock {\em Foundations of Computational Mathematics}, 19(5):1145--1190,
  2019.

\bibitem{lugosi2019near}
G.~Lugosi and S.~Mendelson.
\newblock Near-optimal mean estimators with respect to general norms.
\newblock {\em Probability Theory and Related Fields}, 175:957--973, 2019.

\bibitem{lugosi2019risk}
G.~Lugosi and S.~Mendelson.
\newblock Risk minimization by median-of-means tournaments.
\newblock {\em Journal of the European Mathematical Society}, 22(3):925--965,
  2019.

\bibitem{lugosi2021robust}
G.~Lugosi and S.~Mendelson.
\newblock Robust multivariate mean estimation: The optimality of trimmed mean.
\newblock {\em The Annals of Statistics}, 49(1):393--410, 2021.

\bibitem{massart2000constants}
P.~Massart.
\newblock About the constants in talagrand's concentration inequalities for
  empirical processes.
\newblock {\em The Annals of Probability}, 28(2):863--884, 2000.

\bibitem{mendelson2019unrestricted}
S.~Mendelson.
\newblock An unrestricted learning procedure.
\newblock {\em Journal of the ACM}, 66(6):1--42, 2019.

\bibitem{mendelson2021approximating}
S.~Mendelson.
\newblock Approximating ${L}_p$ unit balls via random sampling.
\newblock {\em Advances in Mathematics}, 386:1--20, 2021.

\bibitem{mendelson2012generic}
S.~Mendelson and G.~Paouris.
\newblock On generic chaining and the smallest singular value of random
  matrices with heavy tails.
\newblock {\em Journal of Functional Analysis}, 262(9):3775--3811, 2012.

\bibitem{mendelson2014singular}
S.~Mendelson and G.~Paouris.
\newblock On the singular values of random matrices.
\newblock {\em Journal of the European Mathematical Society}, 16(4):823--834,
  2014.

\bibitem{mendelson2020robust}
S.~Mendelson and N.~Zhivotovskiy.
\newblock Robust covariance estimation under ${L}_4-{L}_2$ norm equivalence.
\newblock {\em The Annals of Statistics}, 48(3):1648--1664, 2020.

\bibitem{minasyan2023statistically}
A.~Minasyan and N.~Zhivotovskiy.
\newblock Statistically optimal robust mean and covariance estimation for
  anisotropic {G}aussians.
\newblock {\em arXiv preprint arXiv:2301.09024}, 2023.

\bibitem{minsker2022robust}
S.~Minsker and L.~Wang.
\newblock Robust estimation of covariance matrices: Adversarial contamination
  and beyond.
\newblock {\em arXiv preprint arXiv:2203.02880}, 2022.

\bibitem{mourtada2022distribution}
J.~Mourtada, T.~Va{\v{s}}kevi{\v{c}}ius, and N.~Zhivotovskiy.
\newblock Distribution-free robust linear regression.
\newblock {\em Mathematical Statistics and Learning}, 2022.

\bibitem{oliveira2022improved}
R.~I. Oliveira and Z.~F. Rico.
\newblock Improved covariance estimation: optimal robustness and sub-gaussian
  guarantees under heavy tails.
\newblock {\em arXiv preprint arXiv:2209.13485}, 2022.

\bibitem{ostrovskii2019affine}
D.~M. Ostrovskii and A.~Rudi.
\newblock Affine invariant covariance estimation for heavy-tailed
  distributions.
\newblock In {\em Conference on Learning Theory}, pages 2531--2550. PMLR, 2019.

\bibitem{FernandezRico2022}
Z.~F. Rico.
\newblock {\em Optimal statistical estimation: sub-{G}aussian properties,
  heavy-tailed data, and robust- ness}.
\newblock PhD thesis, Instituto de Matem\'atica Pura e Aplicada, June 2022.

\bibitem{rousseeuw1987robust}
P.~J. Rousseeuw and A.~M. Leroy.
\newblock {\em Robust Regression and Outlier Detection}, volume 589.
\newblock John Wiley \& Sons, 2005.

\bibitem{srivastava2013covariance}
N.~Srivastava and R.~Vershynin.
\newblock Covariance estimation for distributions with $2+\varepsilon$ moments.
\newblock {\em The Annals of Probability}, 41(5):3081--3111, 2013.

\bibitem{stigler1973asymptotic}
S.~M. Stigler.
\newblock The asymptotic distribution of the trimmed mean.
\newblock {\em The Annals of Statistics}, pages 472--477, 1973.

\bibitem{Talagrand2014}
M.~Talagrand.
\newblock {\em Upper and Lower Bounds for Stochastic Processes: Modern Methods
  and Classical Problems}, volume~60.
\newblock Springer Science \& Business Media, 2014.

\bibitem{tikhomirov2018sample}
K.~Tikhomirov.
\newblock Sample covariance matrices of heavy-tailed distributions.
\newblock {\em International Mathematics Research Notices},
  2018(20):6254--6289, 2018.

\bibitem{tropp15}
J.~A. Tropp.
\newblock An {I}ntroduction to {M}atrix {C}oncentration {I}nequalities.
\newblock {\em Foundations and Trends in Machine Learning}, 8(1-2):1--230,
  2015.

\bibitem{tukey1963less}
J.~W. Tukey and D.~H. McLaughlin.
\newblock Less vulnerable confidence and significance procedures for location
  based on a single sample: Trimming/{W}insorization 1.
\newblock {\em Sankhy{\=a}: The Indian Journal of Statistics, Series A}, pages
  331--352, 1963.

\bibitem{van2017structured}
R.~Van~Handel.
\newblock Structured random matrices.
\newblock In {\em Convexity and Concentration}, pages 107--156. Springer, 2017.

\bibitem{vershynin2011approximating}
R.~Vershynin.
\newblock Approximating the moments of marginals of high-dimensional
  distributions.
\newblock {\em The Annals of Probability}, 39(4):1591--1606, 2011.

\bibitem{vershynin2012close}
R.~Vershynin.
\newblock How close is the sample covariance matrix to the actual covariance
  matrix?
\newblock {\em Journal of Theoretical Probability}, 25(3):655--686, 2012.

\bibitem{vershynin_2012}
R.~Vershynin.
\newblock {\em Introduction to the non-asymptotic analysis of random matrices},
  page 210–268.
\newblock Cambridge University Press, 2012.

\bibitem{Vershynin2016HDP}
R.~Vershynin.
\newblock {\em High-Dimensional Probability: An Introduction with
  Applications}, volume~47 of {\em Cambridge Series in Statistical and
  Probabilistic Mathematics}.
\newblock Cambridge University Press, 2016.

\bibitem{zhivotovskiy2021dimension}
N.~Zhivotovskiy.
\newblock Dimension-free bounds for sums of independent matrices and simple
  tensors via the variational principle.
\newblock {\em arXiv preprint arXiv:2108.08198}, 2021.

\end{thebibliography}
}
\end{document}